\documentclass[a4paper,12pt]{amsart}

\usepackage{graphicx}
\usepackage{amsthm}
\usepackage{amsfonts}
\usepackage{amssymb}
\usepackage{amsbsy}
\usepackage{graphics}
\usepackage{stmaryrd}
\usepackage{amsfonts}
\usepackage{setspace}
\usepackage{mathtools}
\usepackage{amsmath,enumerate,color,nccmath,geometry,bm}
\usepackage{enumitem}
\usepackage{tikz}

\setlist[enumerate]{topsep=3mm,parsep=0pt,partopsep=0mm,itemsep=1mm,leftmargin=12mm,labelsep=2mm} 
 \numberwithin{equation}{section} 

\newtheorem{theorem}{Theorem}[section]  
\newtheorem{proposition}[theorem]{Proposition}
\newtheorem{lemma}[theorem]{Lemma}
\newtheorem{corollary}[theorem]{Corollary}

\theoremstyle{definition}  
\newtheorem{definition}[theorem]{Definition}
\newtheorem{notation}[theorem]{Notation}
\newtheorem{example}[theorem]{Example}

\newtheorem{remark}[theorem]{Remark}

\newcommand{\R}{\mathbb{R}}
\newcommand{\C}{\mathbb{C}}
\newcommand{\Z}{\mathbb{Z}}
\newcommand{\N}{\mathbb{N}}

\newcommand\T{{\mathbb T}}  
\newcommand\D{\mathbb D} 
\newcommand\cA{{\mathcal A}}  
\newcommand\cB{{\mathcal B}}  
\newcommand\cF{\mathcal{F}}

\renewcommand{\Re}{\text{\normalfont Re\,}} 
\renewcommand{\Im}{\text{\normalfont Im\,}} 
\renewcommand{\epsilon}{\varepsilon}  
\newcommand{\id}{{\rm id}}   

\makeatletter
  \def\mathcomposite{%
     \@ifstar
        {\def\@mathcomposite@option{%
            \baselineskip\z@skip\lineskiplimit-\maxdimen}%
         \@mathcomposite}%
        {\let\@mathcomposite@option\offinterlineskip
         \@mathcomposite}}
  \def\@mathcomposite{%
     \@ifnextchar[\@@mathcomposite{\@@mathcomposite[0]}}
  \def\@@mathcomposite[#1]#2#3#4{%
     #2{\mathchoice
        {\@mathcomposite@{#1}{#3}{#4}\displaystyle{1}}%
        {\@mathcomposite@{#1}{#3}{#4}\textstyle{1}}%
        {\@mathcomposite@{#1}{#3}{#4}%
         \scriptstyle\defaultscriptratio}%
        {\@mathcomposite@{#1}{#3}{#4}%
         \scriptscriptstyle\defaultscriptscriptratio}}}
  \def\@mathcomposite@#1#2#3#4#5{%
     \vcenter{\m@th\@mathcomposite@option
        \dimen@\f@size\p@\dimen@#1\dimen@\dimen@#5\dimen@
        \divide\dimen@ 18
        \edef\@mathcomposite@skipamount{\the\dimen@}%
        \ialign{\hfil$#4##$\hfil\cr
           #2\crcr
           \noalign{\vskip\@mathcomposite@skipamount}%
           #3\crcr}}}
  \makeatother

\newcommand{\utimes}{\kern-0.1ex\mathcomposite[-12]{\mathrel}{\cup}{\times}}                         
\newcommand{\sutimes}{\kern-0.1ex\mathcomposite[-13.2]{\mathrel}{\cup}{\times}\kern-0.1ex}    

\begin{document}

\title[Additive processes on the unit circle and Loewner chains]{Additive processes on the unit circle and Loewner chains}
\author[T. Hasebe]{Takahiro Hasebe}
\address{Department of Mathematics, Hokkaido University, North 10 West 8, Kita-Ku, Sapporo 060-0810, Japan}
\email{thasebe@math.sci.hokudai.ac.jp}
\author[I. Hotta]{Ikkei Hotta}
\address{Department of Applied Science, Yamaguchi University 2-16-1 Tokiwadai, Ube 755-8611, Japan}
\email{ihotta@yamaguchi-u.ac.jp}
\date{\today}

\begin{abstract}
This paper defines the notion of generators for a class of decreasing radial Loewner chains which are only continuous with respect to time. For this purpose,  ``Loewner's integral equation'' which generalizes Loewner's differential equation is defined and analyzed. The definition of generators is motivated by the L\'evy-Khintchine representation for additive processes on the unit circle. Actually, we can and do introduce a homeomorphism between the above class of Loewner chains and the set of the distributions of increments of additive processes equipped with suitable topologies. On the other hand, from the viewpoint of non-commutative probability theory, the above generators also induce bijections with some other objects: in particular, monotone convolution hemigroups and free convolution hemigroups.  Finally, the generators of Loewner chains constructed from free convolution hemigroups via subordination are computed. 
\end{abstract}

\maketitle

\tableofcontents

\section{Introduction}

\subsection{Background and overview of the main results}\label{sec:main}
A Loewner chain is a family of holomorphic mappings that describe an evolution of continuously increasing or decreasing family of domains in the complex plane, and is typically governed by a first order differential equation. Loewner chains had/have the main applications to the Bieberbach conjecture \cite{BDDM86} and SLE (Schramm-Loewner evolution) \cite{Lawler05}, where different kinds of Loewner chains have appeared. One kind appearing in the Bieberbach conjecture is called the radial Loewner chains, where the Denjoy-Wolff points, some special attractive fixed points of mappings of the Loewner chain, appear inside the domain. The other kind used for SLE is called the chordal Loewner chains, where the Denjoy-Wolff points appear on the boundary of the domain. 
A general theory of Loewner chains has recently been established in \cite{BCD2012, CDMG10}, which finally unified those classes of Loewner chains and their differential equations.  

On the other hand, in the context of (non-commutative) probability theory, bijections between the following three objects have been formulated in \cite{FHS18} (see \cite{Jek20,Sch17} for related works): 

\begin{enumerate}[label=\rm(\Roman*)]
\item\label{item:intro1} multiplicative Loewner chains on the unit disk (see Section \ref{sec4}), 
\item unitary multiplicative processes of monotonically independent increments, 
\item\label{item:intro3} monotonically homogeneous Markov processes on $\T$. 
\end{enumerate} 
There are various notions in each field: for example there are natural Loewner chains in terms geometry of its ranges (e.g.\ slit, starlike or convex domains), and there are various probabilistic notions on Markov processes. One of the motivations of \cite{FHS18} was to investigate how those notions in different fields can be interpreted from the viewpoint of those bijective correspondences. 

The main goal of the present paper is to establish a bijection from \ref{item:intro1} to a yet another object: 

\begin{enumerate}[label=\rm(\Roman*)]
\setcounter{enumi}{3}
\item\label{item:intro4} additive processes on $\T$ (see Section \ref{sec:additive}).   
\end{enumerate} 
This bijection is formulated in terms of a certain time-dependent infinitesimal generator. For additive processes, the existence of the generator is known as the L\'evy-Khintchine representation; on the other hand for Loewner chains, there is a notion of Herglotz vector fields which can be interpreted as a time-dependent infinitesimal generator. It looks natural to identify those generators and define a bijection between \ref{item:intro1} and \ref{item:intro4}; see the discussions in Section \ref{sec:monotone} for further details.  However, there is a difficulty: Herglotz vector fields are available only for differentiable (more precisely, absolutely continuous) Loewner chains with respect to time parameter, but the Loewner chains appearing in \ref{item:intro1} are, in general, only continuous with respect to time parameter, which is beyond the scope of the existing theory \cite{CDMG10}.

Therefore, the main issue is how to define a suitable generator for multiplicative Loewner chains. We will settle this issue by introducing ``Loewner integral equations'' which generalize differential equations. 
After all, the main results of this paper can be summarized in the following figure:

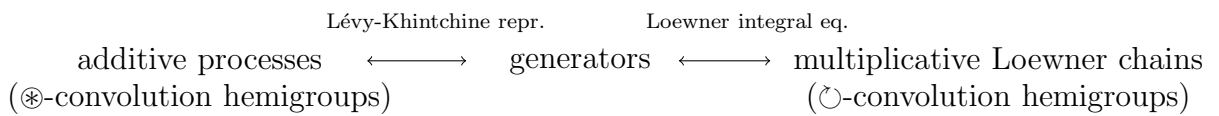
\begin{figure}[h]
\begin{tikzpicture}
\draw[<->] (1.7,0) -- (3.0,0); 
\draw[<->] (5.8,0) -- (7,0); 
\node at (-0.5,0) {additive processes};
\node at (-0.5,-0.5) {($\circledast$-convolution hemigroups)};
\node at (4.5,0) {generators};
\node at (10,0) {multiplicative Loewner chains};
\node at (10,-0.5) {($\circlearrowright$-convolution hemigroups)};
\node at (2.6,0.5)  {\tiny L\'evy-Khintchine repr.\ }; 
\node at (6.7,0.5)  {\tiny Loewner integral eq.\ }; 
\end{tikzpicture}
\caption{The main results}\label{fig:star}
\end{figure}

The appropriateness of the definition of generators can be confirmed from  convergence results: we introduce certain topologies on the three sets in Fig.\ \ref{fig:star} and show that the bijections are homeomorphisms. For example, the topology for Loewner chains is locally uniform convergence on time and space.

We will actually consider the set of certain equivalence classes of additive processes, which is in bijection with the set of continuous convolution hemigroups of probability measures with respect to the natural convolution, denoted by $\circledast$, on the unit circle. On the other hand, the set of multiplicative Loewner chains is in bijection with the set of continuous convolution hemigroups with respect to multiplicative monotone convolution, denoted by $\circlearrowright$, which describe the marginal laws of unitary multiplicative processes of monotonically independent increments in non-commutative probability theory. Further details can be found in Sections \ref{sec2} and \ref{sec4} and also in \cite{FHS18}. 

It is worth mentioning that our idea should work for additive processes on the real line and ``additive Loewner chains'' on the upper half-plane, at least under some extra assumptions. This will be discussed in another paper.

\subsection{Further backgrounds: limit theorems in non-commutative probability}
Bercovici and Pata formulated and proved an equivalence of limit theorems for convolutions of identical probability measures in classical probability and free probability \cite{BP99}. More precisely, let $\{\mu_n\}_{n\in \N}$ be a sequence of probability measures on $\R$ and $\{k_n\}_{n\in\N}$ be a sequence of strictly increasing natural numbers. Then the sequence 
\[
\lambda_n:=\underbrace{\mu_n \ast \mu_n \ast \cdots \ast \mu_n}_{\text{$k_n$ fold}} 
\]
converges weakly to some probability measure on $\R$ if and only if the same holds for 
\[
\tilde\lambda_n:=\underbrace{\mu_n \boxplus \mu_n \boxplus \cdots \boxplus \mu_n}_{\text{$k_n$ fold}}. 
\]
Moreover, if one (and hence both) of these convergences holds, the limit distributions of $\lambda_n$ and $\tilde\lambda_n$ are infinitely divisible and freely infinitely divisible, respectively. 
This leads to a bijection, called the Bercovici-Pata bijection, between the set of infinitely divisible distributions and the set of freely infinitely divisible distributions. Similar results were  established between classical and boolean probabilities as well in the same paper \cite{BP99} and between classical and monotone probabilities in \cite{AW14}.

 A generalization of Bercovici-Pata's theorem to non-identical probability measures was established by Chistyakov and Goetze in \cite{CG08a} as follows. Let $\{\mu_{n,k}\}_{k_n \ge k \ge 1, n \ge 1}$ be an infinitesimal triangular array of probability measures on $\R$, that is,  $\{k_n\}_{n\ge1}$ is a sequence of strictly increasing natural numbers as before, and 
 \begin{equation*}\label{eq:IA0}
 \lim_{n\to\infty}\sup_{1 \le k \le k_n} \mu_{n,k}(\{x: |x|\ge\epsilon\})=0
\end{equation*}
 for all $\epsilon>0$.  Let $\{a_n\}_{n\ge1}$ be a sequence of real numbers. Then 
 \begin{equation}\label{eq:CIA}
  \nu_n:=\delta_{a_n} \ast \mu_{n,1} \ast \mu_{n,2} \ast \cdots \ast \mu_{n,k_n}
\end{equation}
 converges weakly to some probability measure if and only if the same holds for 
\begin{equation}\label{eq:FIA}
\tilde\nu_n:= \delta_{a_n} \boxplus \mu_{n,1} \boxplus \mu_{n,2} \boxplus \cdots \boxplus \mu_{n,k_n}, 
\end{equation}
and if one of those convergences holds then the limits of $\nu_n$ and $\tilde\nu_n$ are related by the Bercovici-Pata bijection. A similar equivalence for limit theorems was established between free and boolean convolutions by Wang \cite{Wan08}. However, the limit theorem for monotone convolution is not equivalent. The main cause is that the limit distribution of monotone convolution of an infinitesimal triangular array is not monotonically infinitely divisible in general; see \cite[Proposition 6.37]{FHS18}. 

By contrast, for probability measures on the unit circle, limit theorems for infinitesimal triangular arrays in classical and free probabilities are not equivalent; one reason is that a L\'evy measure of an infinitely divisible distribution is not unique in general, see \cite[Remark 3, Chapter IV]{Par67} and \cite{CG08b,Ceb16,HH}. Accordingly, a Bercovici-Pata type bijection is not well defined. On the other hand, the L\'evy measure of any freely/boolean infinitely divisible distribution is unique, and correspondingly the equivalence of limit theorems holds true between free and boolean \cite{Wan08}. The limit theorem for monotone convolution is not equivalent to any other one of classical, free and boolean cases; the main cause is the same as the real line case: the limit distribution may not be monotonically infinitely divisible  \cite[Theorem 7.5]{FHS18} (an explicit counterexample may be constructed by imitating the arguments of \cite[Remark 6.15]{FHS18}).

Overall, one cannot hope for an equivalence of limit theorems for infinitesimal triangular arrays between classical and monotone probabilities, either on $\R$ or $\T$. The main results of this paper (that is, Figure \ref{fig:star}) can be regarded as an attempt to overcome this situation by switching to a ``dynamical viewpoint''; we consider all the marginal distributions of a process, rather than looking at a single marginal distribution of a process as in \eqref{eq:CIA} and \eqref{eq:FIA}.  
In this dynamical approach, for additive processes or convolution hemigroups, we can also recover the uniqueness of (time-dependent) L\'evy measures on $\T$ (Theorem \ref{thm:LK}) and, as a consequence, a bijection can be established between different kinds of convolution hemigroups.

\subsection{Organization of the paper}
Section \ref{sec2} starts by introducing basic notions, and then reformulates and summarizes the L\'evy-Khintchine representation of additive processes (or $\circledast$-convolution hemigroups) on the unit circle (Theorem \ref{thm:LK}). As a new result (as far as the authors know),  the convergence of $\circledast$-convolution hemigroups will be characterized by the convergence of the associated generators (Theorem \ref{thm:conv_classical}). For this purpose we introduce the notion of locally uniform weak convergence of a family of probability/finite measures.  

Section \ref{sec4} is the main part of this paper, where we introduce the Loewner integral equation (Section \ref{sec:LIE}), the bijection between the sets of multiplicative Loewner chains and additive processes (Section \ref{sec:gen}), and prove the equivalence of convergence of Loewner chains, the convergence of generators and the convergence of $\circlearrowright$-convolution hemigroups (Section \ref{sec:conv}). 

Section \ref{sec:c-f} is devoted to readers interested in free convolution and boolean convolution. We add two more objects to Fig.\ \ref{fig:star}: convolution hemigroups with respect to multiplicative free convolution and boolean one, and hence five objects are in bijection. We also prove all those bijections are homeomorphic with respect to suitable topologies.  At the end, the generators for the monotone convolution hemigroups associated with free convolution hemigroups are computed. 

Appendices \ref{sec:measure} and \ref{sec:univalent} in the end of this paper summarize several needed facts, which are more or less known but hard to find in the literature.

\subsection*{Acknowledgement} 
T.H.\ is supported by JSPS Grant-in-Aid for Early-Career Scientists 19K14546. 
I.H.\ is supported by JSPS Grant-in-Aid for Scientific Research(C) 20K03632.
This research is an outcome of Joint Seminar supported by JSPS and CNRS under the Japan-France Research Cooperative Program. 
This work was supported by JSPS Open Partnership Joint Research Projects grant no. JPJSBP120209921.
The authors express sincere thanks to Kouji Yano for finding errors of the manuscript, and Pavel Gumenyuk, Takuya Murayama, Sebastian Schlei{\ss}inger and Hiroshi Yanagihara for valuable suggestions and comments.

\section{Additive processes and convolution hemigroups on the unit circle}\label{sec2}

\subsection{Preliminaries on measure theory}\label{sec:measure0}

We summarizes notions on (Borel) measures. Basic results are summarized in Appendix \ref{sec:measure}. 

In this section, let $S$ be a metric space; we will actually take $S=\T$ or $S=\T \times [0,\infty)$ in  later applications. Let $\cB(S)$ be the set of Borel subsets of $S$ and $\cB_b(S)$ the set of bounded Borel subsets of $S$. A \emph{non-negative measure} on $S$ is a $[0,\infty]$-valued $\sigma$-additive function on $\cB(S)$ such that its evaluation at the empty set is zero.
A non-negative measure is referred to as: \emph{locally finite} if it takes finite values on $\cB_b(S)$; \emph{finite} if the total mass is finite; a \emph{probability measure} if the total mass is 1. 
A \emph{locally finite complex} (resp.\ \emph{signed}) \emph{measure} on $S$ is a $\C$-valued (resp.\ $\R$-valued) $\sigma$-additive function on $\cB_{b}(S)$.

\begin{definition} Let $\rho, \rho_n, n=1,2,3,\dots$ be  locally finite non-negative measures on $S$. We say that $\rho_n$ \emph{converges vaguely} to $\rho$ if 
\begin{equation}\label{eq:vague}
\lim_{n\to\infty} \int_S f(x) \,d\rho_n(x) =  \int_S f(x) \,d\rho(x) 
\end{equation}
for every bounded continuous function $f\colon S \to \C$ with bounded support. Moreover, if $\rho,\rho_n$ are finite non-negative measures on $S$, then we say that $\rho_n$ \emph{converges weakly} to $\rho$ if \eqref{eq:vague} holds for every bounded continuous function $f\colon S \to \C$. 
\end{definition}

\begin{remark} The above definitions of locally finite measures and vague convergence are taken from \cite{Kal,BP19}, but the definitions depend on the literature.  According to \cite{Bau01}, a non-negative measure is locally finite if every point of $S$ has an open neighborhood of finite mass with respect to this measure, and the vague convergence is defined by requiring \eqref{eq:vague} for all continuous functions with compact support. 
When all the closed bounded subsets are compact (which is valid for $S=\T$ and $S=\T \times [0,\infty)$), then the two definitions are equivalent. 
\end{remark}

We conclude this section by introducing a suitable notion of convergence for a sequence of families of finite non-negative measures. This will be a key concept throughout this paper.  
 
 \begin{definition}\label{def:local_weak} Let $T$ be a metric space and $(\tau_t^{(n)})_{t\in T},(\tau_t)_{t\in T}, n=1,2,3,\dots$ be families of finite non-negative measures on $S$. We say that $(\tau_t^{(n)})_{t\in T}$ \emph{weakly converges to $(\tau_t)_{t\in T}$ locally uniformly} on $T$ as $n\to \infty$ if 
 \begin{equation}\label{eq:local_weak}
\lim_{n\to\infty}\sup_{t \in B}\left|\int_S  f(x) \,d\tau_{t}^{(n)}(x)  -\int_S f(x) \,d\tau_{t}(x) \right|=0 
\end{equation}
for every bounded subset $B$ of $T$ and every bounded continuous function $f\colon S \to \C$.  
 \end{definition}
 In this paper, the index set $T$ is either $[0,\infty)$ or $\{(s,t): 0 \le s \le t <\infty\}$.

\subsection{L\'evy-Khintchine representation for additive processes} \label{sec:additive}

We will focus on the group $\T$, the group of complex numbers with modulus one,  but many facts in this section can be extended to locally compact abelian groups with countable basis of its topology. The main reference of this section is \cite{Heyer}. 

Two stochastic processes $(Y_t)_{t\ge0}$ and $(\tilde Y_t)_{t\ge0}$ on $\T$ are said to be \emph{stochastically equivalent} if 
\begin{equation}\label{eq:equiv}
\mathbb P(Y_{t_1} \in B_1, Y_{t_2}\in B_2, \dots, Y_{t_n} \in B_n) =\mathbb P(\tilde Y_{t_1} \in B_1, \tilde Y_{t_2}\in B_2, \dots, \tilde Y_{t_n} \in B_n)
\end{equation}
for all $n\in \N, t_1, \dots, t_n \ge0$ and $B_1, B_2,\dots, B_n \in \cB(\T)$. 

A stochastic process $(Y_t)_{t\ge0}$ on $\T$ is called an \emph{additive process} if $Y_0=1$ a.s., it has independent increments, that is, $Y_{t_{0}}^{-1} Y_{t_1}, Y_{t_1}^{-1} Y_{t_2}, \dots, Y_{t_{n-1}}^{-1} Y_{t_n}$ are independent for all $0 \leq t_0 \leq t_1 \le \cdots \le t_n$ and $n\in \N$, the mapping  $(s,t)\mapsto Y_{s}^{-1}Y_t$ is stochastically continuous, and its sample path is almost surely c\`adl\`ag, that is, right continuous with finite left limits. An additive process $(\tilde Y_t)_{t\ge0}$ is called a \emph{L\'evy process} if the distribution of $Y_s^{-1}Y_t$ depends only on $t-s$ for every $0 \le s \le t$.

There is a natural binary operation on the set of probability measures on $\T$, denoted by $\circledast$, arising from the group structure, but other associative binary operations from non-commutative probability will also be discussed.  Hence for generality, let $\star$ be an associative binary operation (convolution) on the set of Borel probability measures on $\T$. A family of probability measures $(\mu_{s,t})_{t\ge s\ge0 }$ on $\T$ is called a \emph{$\star$-convolution hemigroup} ($\star$-CH for short) if $\mu_{t,t}=\delta_1$ for every $t\ge0$ and $\mu_{s,t} \star \mu_{t,u}=\mu_{s,u}$ for every $0 \le s \le t \le u$.  It is furthermore said to be \emph{continuous} if $(s,t)\mapsto \mu_{s,t}$ is weakly continuous. If a $\star$-CH satisfies the time-homogeneity, $\mu_{s,t}=\mu_{0,t-s}$ for all $t\ge s\ge0 $, then the probability measures $\mu_t:=\mu_{0,t}$ form a $\star$-convolution semigroup, namely, one has $\mu_{s}\star\mu_t = \mu_{s+t}$ for all $s,t\ge0$ and $\mu_0=\delta_1$. 

Given an additive process $(Y_t)_{t\ge0}$ on $\T$ the distributions $\mu_{s,t}  = \mathbb P_{Y_s^{-1}Y_t}$ form a continuous $\circledast$-CH. Conversely, for any continuous $\circledast$-CH $(\mu_{s,t})_{t\ge s\ge0 }$ on $\T$ there exists an additive process $(Y_t)_{t\ge0}$ such that $\mu_{s,t}  = \mathbb P_{Y_s^{-1}Y_t}$ \cite[Remark 5.6.2]{Heyer} by taking the canonical coordinate process associated with the Markov transition kernels $k_{s,t}(x,\cdot):=\delta_x \circledast  \mu_{s,t}$ and then constructing its c\`adl\`ag modification by Dynkin--Kinney's theorem; the reader can also find detailed arguments in \cite[Theorems 9.7(ii) and 11.5]{Sat13} for the case of Euclidean spaces, which can be adjusted to our case $\T$. Note that this additive process is unique up to stochastic equivalence because the quantity \eqref{eq:equiv} can be computed only by $(\mu_{s,t})_{t\ge s\ge0 }$.  

The characteristic function of an additive process has a \emph{L\'evy-Khintchine representation} as follows. The results are more or less known, but some statements  are not very straightforward from the literature, so a proof is provided. 
   
\begin{theorem}[L\'evy-Khintchine representation] \label{thm:LK}
For a continuous $\circledast$-CH $(\mu_{s,t})_{t\ge s \ge 0}$ on $\T$ there exists a unique family $(\alpha_t,\sigma_t)_{t\ge0}$ such that 
\begin{enumerate}[label=\rm(LK\arabic*)]
\item   \label{LK1} $t\mapsto \alpha_t \in \R$ is a continuous function such that $\alpha_0=0$, 
\item\label{LK2} $\sigma_t$ is a finite non-negative measure on $\T$ such that $\sigma_0=0$ and the function $t\mapsto \sigma_t(B)$ is continuous and non-decreasing for every $B \in \cB(\T)$,    
\end{enumerate}
and 
\begin{equation}\label{eq:LKT}
\int_\T \xi^n \,d\mu_{0,t}(\xi) = \exp \left(i \alpha_t n+ \int_\T \frac{\xi^n-1- in  \Im(\xi)}{1-\Re(\xi)}\,\sigma_t(d\xi)  \right), \qquad n \in \Z.  
\end{equation}
Conversely, given a family $(\alpha_t,\sigma_t)_{t\ge0}$ satisfying \ref{LK1} and \ref{LK2} there exists a unique continuous $\circledast$-CH $(\mu_{s,t})_{t\ge s \ge 0}$ on $\T$ for which \eqref{eq:LKT} holds. The family $(\alpha_t,\sigma_t)_{t\ge0}$ is called the \emph{generating family} for the corresponding continuous $\circledast$-CH. 
\end{theorem}
\begin{remark}
For each $n\in \Z$ the function 
\[
\xi\mapsto  \frac{\xi^n-1- in  \Im(\xi)}{1-\Re(\xi)} 
\]
is set to be $-n^2$ at $\xi=1$, which makes it continuous on $\T$. 
\end{remark}
\begin{proof} It is proved in \cite[Theorem 5.6.19]{Heyer} that for every $t\ge0$ there exist $\alpha_t\in\R, v_t \ge0$ and a non-negative measure $\pi_t$ on $\T \setminus\{1\}$ such that $\int_\T (1-\Re(\xi))d\pi_t(\xi)<\infty$ and 
\begin{equation}\label{eq:standardLK}
\int_\T \xi^n \,d\mu_{0,t}(\xi)= \exp \left(i \alpha_t n - \frac{v_t}{2}n^2+ \int_\T \left(\xi^n-1- in  \Im(\xi)\right)\pi_t(d\xi)  \right), \qquad n \in \Z.  
\end{equation}
Denote by $(\tilde Y_t)_{t\ge0}$ an additive process corresponding to $(\mu_{s,t})$. According to \cite[Preparations 5.6.5]{Heyer}, for each $B \in \cB(\T)$ away from 1 (namely, the distance between $B$ and 1 is positive),  $\pi_t(B)$ is defined to be the expectation of the number of $s \in [0,t]$ such that $Y_s Y_{s-} \in B$. 
The non-negative finite measure 
\[
d\sigma_t(\xi) = \frac{1}{2}v_t \delta_1(d\xi) + (1-\Re(\xi))d\pi_t|_{\T\setminus\{1\}}(\xi)
\] 
on $\T$ then satisfies \eqref{eq:LKT}. Now we prove \ref{LK1} and \ref{LK2}. Since $e^{i\alpha_t} = \mathbb E[Y_t]/|\mathbb E[Y_t]|$ is continuous on $t$ and $Y_0=1$, the function $\alpha_t$ can be taken to be continuous such that $\alpha_0=0$.  The definition of $\pi_t$ implies that $\pi_0=0$, and by \eqref{eq:LKT} $v_0=0$, showing $\sigma_0=0$. 
The definition of $\pi_t$ and the proof of \cite[Lemma 5.6.13]{Heyer} show that the mapping $t\mapsto\pi_t(B)$ is non-decreasing and continuous for every $B \in \cB(\T)$ away from 1, and hence the same holds for $\sigma_t(B)$. Since $\sigma_t(\{1\})= v_t/2$ is non-decreasing on $t$ by \cite[Theorem 5.6.19]{Heyer}, we can see that $t\mapsto \sigma_t(B)$ is non-decreasing for every $B \in \cB(\T)$. It remains to show that $t\mapsto \sigma_t(B)$ is continuous for every $B\in\cB(\T)$. Since $\sigma_t(\T) = -\log |\mathbb E[Y_t]|$ is continuous on $t$, we conclude that $t\mapsto \sigma_t(U)=\sigma_t(\T)-\sigma_t(U^c)$ is continuous as well for every open neighborhood $U$ of $1$. Now fix $t_0>0$. For $0\le s ,t <t_0$ and every open neighborhood $U$ of $1$ the inequality 
\begin{align*}
|\sigma_t(\{1\}) - \sigma_s(\{1\})| 
&\le  |\sigma_t(U) - \sigma_s(U)|  +|\sigma_t(U\setminus\{1\}) - \sigma_s(U\setminus\{1\})| \\
&\le |\sigma_t(U) - \sigma_s(U)|  +2\sigma_{t_0}(U\setminus\{1\}) 
\end{align*}
holds. By taking $U$ small enough and then $t$ close to $s$, we conclude that $\lim_{t\to s} \sigma_t(\{1\}) =\sigma_s(\{1\})$. Therefore it suffices to show the continuity of $\sigma_t(B)$ for every $B \in \cB(\T\setminus\{1\})$. For such $B$ we have 
\begin{align*}
|\sigma_t(B) - \sigma_s(B)| 
&\le |\sigma_t(B \cap U^c) - \sigma_s(B\cap U^c)|   +|\sigma_t(B\cap U) - \sigma_s(B\cap U)| \\
&\le  |\sigma_t(B \cap U^c) - \sigma_s(B\cap U^c)| + 2 \sigma_{t_0}(U\setminus\{1\})
\end{align*}
for every open neighborhood $U$ of 1 and $0\le s,t < t_0$. Similarly to the previous case we conclude that $t\mapsto \sigma_t(B)$ is continuous.

For the uniqueness of $(\alpha_t,\sigma_t)_{t\ge0}$, suppose that $(\tilde\alpha_t,\tilde \sigma_t)_{t\ge0}$ is another family. There must exist a function $f\colon [0,\infty) \times \Z \to \Z$ such that 
\begin{equation}\label{eq:LK-unique}
i \alpha_t n+ \int_\T \frac{\xi^n-1- in  \Im(\xi)}{1-\Re(\xi)}\,\sigma_t(d\xi)=i \tilde\alpha_t n+ \int_\T \frac{\xi^n-1- in  \Im(\xi)}{1-\Re(\xi)}\,\tilde\sigma_t(d\xi) + 2\pi i f(t,n). 
\end{equation}
By the assumptions  \ref{LK1} and \ref{LK2} the function $f(\cdot, n)$ is continuous for every $n\in \Z$, and hence is constant, namely $f(t,n)=f(0,n)=0$. Setting $n=1$ in \eqref{eq:LK-unique} and taking the imaginary part show $\alpha_t=\tilde\alpha_t$. Lastly, by \cite[Propositions 6.2 and 8.2]{HH} we conclude that $\sigma_t =\tilde \sigma_t$. 

Conversely, given a generating family, the corresponding $\circledast$-CH $(\mu_{s,t})_{t\ge s\ge0 }$ exists by \cite[Theorem 7.1, Chapter IV]{Par67}, which is continuous by \ref{LK1} and \ref{LK2}. 
\end{proof}

By the hemigroup property one has the L\'evy-Khintchine representation for the increments 
\begin{equation}\label{eq:increments_c}
\int_\T \xi^n \,d\mu_{s,t}(\xi)= \exp \left(i \alpha_{s,t} n+ \int_\T \frac{\xi^n-1- in  \Im(\xi)}{1-\Re(\xi)}\,\sigma_{s,t}(d\xi)  \right), \quad n \in \Z,~ 0 \le s \le t,    
\end{equation}
where $\alpha_{s,t}= \alpha_t -\alpha_s$ and $\sigma_{s,t}=\sigma_t - \sigma_s$. 

In probability theory \eqref{eq:standardLK} is more common than \eqref{eq:LKT}. One advantage of \eqref{eq:standardLK} is that $\pi_t$, called a L\'evy measure, has a direct probabilistic interpretation by its definition. Moreover, \eqref{eq:standardLK} is much easier to generalize to other locally compact abelian groups. Nevertheless, we will use \eqref{eq:LKT} which is more suitable to compare with the integral representation of Herglotz functions.

In \eqref{eq:LKT} the family of measures $(\sigma_t)_{t\ge0}$ associates the non-negative measure $\Sigma$ on $\T \times [0,\infty)$ defined by 
\begin{equation}\label{eq:sigma}
\Sigma(B\times (s,t]) = \sigma_t(B) - \sigma_s(B)
\end{equation}
 for all $t > s \ge 0$ and $B\in \cB(\T)$ (see the proof of Proposition \ref{bi-measure} below). The non-negative measure $\Sigma$ obtained in this way is characterized by the following property: 
\begin{enumerate}[label=\rm(LK3)]
\item\label{LK3} $\Sigma$ is a locally finite non-negative measure on $\T \times [0,\infty)$ such that $\Sigma(\T\times\{t\})=0$ for every $t\ge0$.     
\end{enumerate}

 We may thus employ the parameterization $((\alpha_t)_{t\ge0},\Sigma)$ instead of $(\alpha_t,\sigma_t)_{t\ge0}$.

\subsection{Convergence of continuous $\circledast$-convolution hemigroups} 
 
 This section characterizes the locally uniform weak convergence of convolution hemigroups in terms of convergence of generating families. 
 
\begin{proposition} \label{prop:convergence_equivalence} Suppose that time-dependent finite non-negative measures $(\sigma_t)_{t\ge0}$, $(\sigma_t^{(n)})_{t\ge0}, n=1,2,3,\dots$ satisfy \ref{LK2}. Let $\Sigma^{(n)}$ and $\Sigma$ be the non-negative measures associated to $(\sigma_{t}^{(n)})_{t\ge0}$ and $ (\sigma_{t})_{t\ge0}$ via \eqref{eq:sigma}, respectively.  The following assertions are equivalent as $n\to\infty$: 
\begin{enumerate}[label=\rm(G\arabic*)]
\item\label{condG1} $(\sigma_t^{(n)})_{t\ge0}$ weakly converges to $(\sigma_t)_{t\ge0}$ locally uniformly on $[0,\infty)$; 

\item\label{condG2}  $\sigma_t^{(n)}$ converges weakly to $\sigma_t$ for every $t\ge0$; 

\item\label{condG3} $\Sigma^{(n)}$ converges vaguely to $\Sigma$; 

\item\label{condG4} $\Sigma^{(n)}|_{\T \times [0,T]}$ converges weakly to $\Sigma|_{\T \times [0,T]}$ as finite measures on $\T \times [0,T]$  for every $T >0$. 

\end{enumerate} 

\end{proposition}
\begin{proof} Clearly \ref{condG1} implies \ref{condG2}, and \ref{condG4} implies \ref{condG3}. 
We proceed as \ref{condG3} $\Rightarrow$ \ref{condG4}, \ref{condG2} $\Leftrightarrow$ \ref{condG3} and \ref{condG2}+\ref{condG4} $\Rightarrow$ \ref{condG1}.

\vspace{1mm}\noindent
\underline{\ref{condG3} $\Rightarrow$ \ref{condG4}.} 
Let $T>0$ and $g$ be a continuous (and hence bounded) function on $\T \times [0,T]$. We approximate the function $g$ by 
\[
g_k(\xi, t) = 
\begin{cases} 
g(\xi,t), & t \in [0,T], \\  
\max\{-k (t- T) + 1,0\} g(\xi,T), & t>T. 
\end{cases}  
\] 
Let $C:= \sup_{t\in [0,T], \xi\in \T } |g(\xi,t)|$. Note that $g_k$ is supported on $\T \times [0,T+1/k]$ and $|g_k| \le C$. By \ref{LK3} for $\Sigma$ and Proposition \ref{prop:portmanteau} \ref{item:conv3} for every $\epsilon>0$ there exists a $\delta>0$ such that 
\[
\Sigma(\T \times [T,T+\delta]) <\epsilon  \qquad \text{and} \qquad  \sup_{n\in\N} \Sigma^{(n)}(\T \times [T,T+\delta]) <\epsilon. 
\]
Then, for every $k \ge \delta^{-1}$ we have 
\[
\sup_{n\in \N}\left| \int_{\T \times [0,T]} g(\xi,t) \,\Sigma^{(n)}(d\xi dt) - \int_{\T \times [0,\infty)} g_k(\xi,t) \,\Sigma^{(n)}(d\xi dt)\right| \le C\epsilon 
\]
and 
\[
\left| \int_{\T \times [0,T]} g(\xi,t) \,\Sigma(d\xi dt) - \int_{\T \times [0,\infty)} g_k(\xi,t) \,\Sigma(d\xi dt)\right| \le C\epsilon. 
\]
By the assumption that $\Sigma^{(n)}$ converges vaguely to $\Sigma$, for a fixed $k \ge \delta^{-1}$ and sufficiently large $n$ we also have
\[
\left| \int_{\T \times [0,\infty)} g_k(\xi,t) \,\Sigma^{(n)}(d\xi dt) - \int_{\T \times [0,\infty)} g_k(\xi,t) \,\Sigma(d\xi dt)\right| \le \epsilon. 
\]
Combining the three estimates we conclude that 
\[
\left| \int_{\T \times [0,T]} g(\xi,t) \,\Sigma^{(n)}(d\xi dt) - \int_{\T \times [0,T]} g(\xi,t) \,\Sigma(d\xi dt)\right| \le (2C+1)\epsilon. 
\]

\vspace{2mm}\noindent\underline{\ref{condG3} $\Rightarrow$ \ref{condG2}.}
Fix $t\ge0$. We use Proposition \ref{prop:portmanteau_finite} \ref{item:conv_finite3}. Suppose that $B \in \cB(\T)$ is such that $\sigma_t(\partial B)=0$.  The relation $\partial (B \times [0,t]) = (\partial B\times [0,t)) \cup ((B\cup\partial B) \times \{t\})$ shows that $\Sigma(\partial (B \times [0,t]))=0$. Hence, thanks to Proposition \ref{prop:portmanteau}, we have
\[
\sigma_t^{(n)}(B) = \Sigma^{(n)}(B \times [0,t]) \to \Sigma(B \times [0,t])= \sigma_t(B)
\] 
as $n\to\infty$. 

\vspace{2mm}\noindent\underline{\ref{condG2} $\Rightarrow$ \ref{condG3}.}  Take a continuous function $f\colon \T \times[0,\infty)\to \C$ with compact support. Let $T>0$ be such that $f$ vanishes outside $\T \times [0,T]$.  The goal is to prove that 
\begin{equation}\label{eq:conv_Sigma}
\int_{\T\times [0,T]} f(\xi,t) \Sigma^{(n)}(d\xi dt) \to \int_{\T \times [0,T]} f(\xi,t) \Sigma(d\xi dt). 
\end{equation}
Let $t_k=t_{k}(\ell)=kT/\ell$ and let $f_\ell(\xi,t)= \sum_{k=0}^{\ell-1} f(\xi,t_k) \mathbf 1_{[t_k,t_{k+1})}(t)$. It is easy to see that $\lim_{\ell \to \infty}\|f_\ell -f \|_{\infty} =0$, and hence in \eqref{eq:conv_Sigma} replacing $f$ with $f_\ell$ for a sufficiently large $\ell$ yields only a small error uniformly on $n$. Therefore, it suffices to estimate
\begin{align}
&\int_{\T\times [0,T]} f_\ell(\xi,t) \Sigma^{(n)}(d\xi dt) - \int_{\T \times [0,T]} f_\ell(\xi,t) \Sigma(d\xi dt)    \label{eq:sigmas}   \\ 
&\qquad= \sum_{k=0}^{\ell-1} \left[\int_{\T} f(\xi,t_k) \,d\sigma^{(n)}_{t_{k+1}}(d\xi) -\int_{\T} f(\xi,t_k) \,d\sigma_{t_{k+1}}(d\xi)\right]   \notag  \\ 
& \qquad\qquad -  \sum_{k=0}^{\ell-1} \left[\int_{\T} f(\xi,t_k) \,d\sigma^{(n)}_{t_{k}}(d\xi) -\int_{\T} f(\xi,t_k) \,d\sigma_{t_{k}}(d\xi)\right]. \notag
\end{align}
The assumption \ref{condG2} implies that the RHS of \eqref{eq:sigmas} converges to 0 as $n\to\infty$, which establishes \eqref{eq:conv_Sigma}. 
 
\vspace{1mm}\noindent
\underline{\ref{condG2} and \ref{condG4}  $\Rightarrow$ \ref{condG1}.} Take $\epsilon, T>0$. By \ref{condG4}, $\Sigma^{(n)}(\T\times \cdot )$ converges weakly to $\Sigma(\T\times \cdot)$ as finite measures on $[0,T]$. Since the measure $\Sigma(\T,\cdot)$ does not have an atom, Polya's theorem \cite[Theorem 4.5]{BW16} implies that $\Sigma^{(n)}(\T\times[0,t])$ converges to $\Sigma(\T\times[0,t])$ uniformly on $[0,T]$. Combining this with the fact that $t\mapsto \Sigma(\T\times[0,t])$ is uniformly continuous on $[0,T]$, we conclude that there exists $\delta>0$ such that  
\[
\Sigma(\T,(s,t]) < \epsilon \qquad \text{and} \qquad \sup_{n\in\N} \Sigma^{(n)}(\T,(s,t]) < \epsilon
\]
for all $s,t \in [0,T]$ with $0 \le t-s \le \delta$. 
Now, take a partition $0=t_0 < t_1 < \cdots < t_N= T$ of the interval $[0,T]$ such that $\sup_{i \in \{1,\dots,N\}}|t_i-t_{i-1}| \le \delta$. By the assumption \ref{condG2} there exists $n_0 \in \N$ such that 
\[
\sup_{\substack {n\ge n_0 \\ 0 \le i \le N } }\left|\int_\T f(\xi)\,\sigma_{t_i}^{(n)}(d\xi) - \int_\T f(\xi)\,\sigma_{t_i}(d\xi)\right| < \epsilon. 
\]
 For $t \in [0,T]$ we take $i$ such that  $t\in [t_i,t_{i-1}]$; then 
\begin{align*}
&\left|\int_\T f(\xi)\,\sigma_{t}^{(n)}(d\xi) - \int_\T f(\xi)\,\sigma_{t}(d\xi)\right| \\
&\le  \left|\int_\T f(\xi)\,\sigma_{t}^{(n)}(d\xi) - \int_\T f(\xi)\,\sigma_{t_i}^{(n)}(d\xi)\right| +  \left|\int_\T f(\xi)\,\sigma_{t_i}^{(n)}(d\xi) - \int_\T f(\xi)\,\sigma_{t_i}(d\xi)\right| \\
&+  \left|\int_\T f(\xi)\,\sigma_{t_i}(d\xi) - \int_\T f(\xi)\,\sigma_{t}(d\xi)\right| \\
&\le \left\{ \Sigma^{(n)}(\T,(t,t_i] )  + \Sigma(\T,(t,t_i] )\right\} \|f\|_{L^\infty}+ \epsilon \\
& \le 2\epsilon  \|f\|_{L^\infty} + \epsilon, 
\end{align*}
whenever $n \ge n_0$. This verifies \ref{condG1}. 
 \end{proof}

We provide the main result of this section. 
 
 \begin{theorem}\label{thm:conv_classical}
Let $(\mu_{s,t})_{t\ge s \ge 0}, (\mu_{s,t}^{(n)})_{t\ge s \ge 0}$, $n=1,2,3,\dots$ be continuous $\circledast$-CHs and $(\alpha_t,\sigma_t)_{t\ge0}$, $(\alpha_t^{(n)},\sigma_t^{(n)})_{t\ge0}$ be their respective generating families. The following conditions are equivalent as $n\to\infty$. 
\begin{enumerate}[label=\rm(C\arabic*)]
\item\label{condC1} $(\mu_{0,t}^{(n)})_{t\ge0}$ weakly converges to $(\mu_{0,t})_{t\ge0}$ locally uniformly on $[0,\infty)$; 

\item\label{condC2} $(\mu_{s,t}^{(n)})_{t\ge s \ge 0}$ weakly converges to $(\mu_{s,t})_{t\ge s \ge 0}$ locally uniformly on the index set $\{(s,t): 0 \le s \le t <\infty\}$;  

\item\label{condC3} $(\alpha_t^{(n)})_{t\ge0}$ converges to $(\alpha_t)_{t\ge0}$ locally uniformly on $[0,\infty)$ and the mutually equivalent conditions \ref{condG1}--\ref{condG4} in Proposition \ref{prop:convergence_equivalence} hold. 
\end{enumerate} 
\end{theorem}
\begin{proof} The implication \ref{condC2} $\Rightarrow$ \ref{condC1} is obvious.

\vspace{2mm}\noindent
\underline{\ref{condC3} $\Rightarrow$ \ref{condC2}.}  By the Stone-Weierstrass theorem, it suffices to consider the test functions $f(\xi)=\xi^k, k\in \Z$ for condition \ref{condC2}. Then the conclusion is immediate from \eqref{eq:increments_c} and \ref{condG1}.

\vspace{2mm}\noindent
\underline{\ref{condC1} $\Rightarrow$ \ref{condC3}.} Fix $T>0$. Since 
\[
 \exp (i \alpha_t^{(n)})= \frac{\int_\T \xi \,d\mu_{0,t}^{(n)}(\xi)}{\left|\int_\T \xi \,d\mu_{0,t}^{(n)}(\xi)\right|}    
 \qquad\text{and}\qquad
   \exp (i \alpha_t)= \frac{\int_\T \xi \,d\mu_{0,t}(\xi)}{\left|\int_\T \xi \,d\mu_{0,t}(\xi)\right|}
\]
the function $\exp (i \alpha_t^{(n)})$ converges to $\exp (i \alpha_t)$ uniformly on $[0,T]$. The same conclusion holds for $\alpha_t^{(n)}$ and $\alpha_t$ since they are continuous and vanish at $t=0$. 

The identities  
\begin{equation}
\left|\int_\T \xi \,d\mu_{0,t}^{(n)}(\xi)\right| = e^{-\Sigma^{(n)}(\T\times[0,t])}
 \qquad\text{and}\qquad
\left|\int_\T \xi \,d\mu_{0,t}(\xi)\right| = e^{-\Sigma(\T\times[0,t])}
\end{equation}
and the assumption \ref{condC1} imply that the function $\Sigma^{(n)}(\T\times[0,t])$ converges to $\Sigma^{(n)}(\T\times[0,t])$ uniformly on $[0,T]$. In particular, the sequence $\{\Sigma^{(n)}|_{\T\times [0,T]}\}_{n\ge1}$ is uniformly bounded and hence has a weakly convergent subsequence by Prohorov's theorem (see Theorem \ref{thm:Prohorov}), the limit finite non-negative measure being denoted by $\tilde \Sigma$. We also set $\tilde \sigma_t(\cdot) = \tilde \Sigma(\cdot, [0,t])$ for $t \in [0,T]$. From the weak convergence and Proposition \ref{prop:portmanteau} \ref{item:conv3}, we deduce that $\Sigma^{(n)}(\T\times[0,t])$ converges to $\tilde \Sigma(\T\times[0,t])$ at all $t\in [0,T]$ where $\tilde \Sigma(\T\times[0,t])$ is continuous. Since $\Sigma^{(n)}(\T\times[0,t])$ is uniformly converging to $ \Sigma(\T\times[0,t])$  and $\tilde \Sigma(\T\times[0,t])$ is right-continuous, we conclude that $\tilde \Sigma(\T\times[0,t]) = \Sigma(\T\times[0,t])$, and hence $(\alpha_t,\tilde \sigma_t)_{t\in [0,T]}$ satisfies \ref{LK1} and \ref{LK2} until time $T$.  Moreover, the proof of Proposition \ref{prop:convergence_equivalence} \ref{condG3} $\Rightarrow$ \ref{condG2} shows that $\sigma_t^{(n)} \to \tilde \sigma_t$ weakly for each $t\in [0,T]$. Passing to the limit $n\to\infty$ in the formula 
\[
\int_\T \xi^k \,d\mu_{0,t}^{(n)}(\xi) = \exp \left(i \alpha_t^{(n)} k+ \int_\T \frac{\xi^k-1- ik  \Im(\xi)}{1-\Re(\xi)}\,\sigma_t^{(n)}(d\xi)  \right), \qquad k \in \Z,~t\in [0,T], 
\]
we obtain 
\[
\int_\T \xi^k \,d\mu_{0,t}(\xi) = \exp \left(i \alpha_t k+ \int_\T \frac{\xi^k-1- ik  \Im(\xi)}{1-\Re(\xi)}\,\tilde \sigma_t(d\xi)  \right), \qquad k \in \Z,~t\in [0,T].  
\]
By the uniqueness of the generating family (Theorem \ref{thm:LK}), we have $\tilde \sigma_t= \sigma_t$ for every $t\in[0,T]$. The above arguments demonstrate that the whole sequence $\{\sigma_t^{(n)}\}_{n\ge1}$ converges weakly to $\sigma_t$ for every $t$ in $[0,T]$ and hence in $[0,\infty)$. 
\end{proof}

\section{Bijection between classical and monotone convolution hemigroups}\label{sec4}
The aim of this section is to define a time-dependent generator for a general multiplicative Loewner chain and then formulate a bijection between additive processes (up to equivalence) and multiplicative Loewner chains, or equivalently, classical and monotone convolution hemigroups. For this purpose we introduce and analyze ``Loewner's integral equation''.

\subsection{Multiplicative Loewner chains and monotone convolution hemigroups}

We start from recalling the definition of multiplicative Loewner chains from \cite[Definition 3.1]{FHS18}.

\begin{definition} Let $D$ be a simply connected domain of $\C$ such that $D \neq \C$. 
A two-parameter family $(f_{s,t})_{0 \le s \le t < \infty}$ of holomorphic self-mappings $f_{s,t} \colon D \to D$ is said to be a \textit{reverse evolution family}  if the following conditions are satisfied;
\begin{enumerate}[label=\rm(TM\arabic*)]
\item \label{TM1} $f_{s,s}(z) =z$ for all $z \in D$ and $s \ge 0$,
\item \label{TM2} $f_{s,t} \circ f_{t,u} = f_{s,u}$ for all $0 \le s \le t \le u < \infty$,
\item \label{TM3} $(s,t)\mapsto f_{s,t}$ is continuous with respect to locally uniform convergence on $D$,   
\end{enumerate} 
and each mapping $f_{s,t}$ is called a \textit{transition mapping}. 
The family $(f_{t})_{t \ge 0} := (f_{0,t})_{t \ge 0}$ is called a \textit{(decreasing) Loewner chain} on $D$. 
Moreover, $(f_t)_{t\ge0}$ is called a \textit{multiplicative Loewner chain} if $D = \D$ and $f_{t}(0)=0$ for all $t \in [0,\infty)$. 
\end{definition}

\begin{remark}
\begin{enumerate}[label=(\alph*)]
\item By \ref{TM2} the inclusion relation $f_{s}(D) \supset f_{t}(D)$ holds for all $s<t$, so that $(f_t)_{t\ge0}$ is called a decreasing Loewner chain.  

\item It was shown in \cite[Theorem 3.16]{FHS18} that all transition mappings $f_{s,t}$, and hence $f_{t}$ as well, are univalent.

\item Under a further regularity assumption, such a family $(f_{s,t})_{0 \le s \le t}$ is an example of a reverse evolution family of order $d$ in Loewner theory \cite{CDMG14}. 

\end{enumerate}
\end{remark}

\begin{notation}
 Sometimes the notation $f(z,t)$ is used for $f_t(z)$. In addition to the standard notation for partial derivatives, we also use the notation $\dot f$ or $\dot f_t$ for the $t$-derivative and $f_t'$ for the $z$-derivative. We will discuss later sequences of Loewner chains, which will be denoted as $(f_t^{(n)})_{t\ge0}$, not to be confused with higher order derivatives. 
\end{notation}

The set of multiplicative Loewner chains is in bijection with the set of continuous convolution hemigroups with respect to multiplicative monotone convolution \cite[Section 5.4]{FHS18} as follows. 

For probability measures $\mu$ and $\nu$ on $\T$, the \emph{multiplicative  monotone convolution} $\mu\circlearrowright \nu$ is the distribution of $U V$ where $U$ and $V$ are unitaries distributed as $\mu$ and $\nu$ respectively, such that $(U-I,V-I)$ is monotonically independent \cite{Ber05}. 
The binary operation $\circlearrowright$ can be characterized by the composition of a certain transform defined as follows.  For a probability measure $\mu$ the \emph{moment generating function} is defined by 
\[
\psi_\mu(z) = \sum_{n\ge1} z^n \int_{\T} \xi^n \,d\mu(\xi) = \int_{\T} \frac{z \xi}{1-z \xi} \,d\mu(\xi), \qquad z \in \D.   
\]
A related function $\eta_\mu$, called the \emph{$\eta$-transform} of $\mu$,  is then defined by 
\[
\eta_\mu(z) =\frac{\psi_\mu(z)}{1+\psi_\mu(z)}, 
\] 
which is a holomorphic self-mapping of $\D$. 
It is proved in \cite{Ber05} that 
\begin{equation}\label{eq:monotone_conv}
\eta_{\mu \circlearrowright \nu} = \eta_\mu \circ \eta_\nu \quad \text{on} \quad \D. 
\end{equation}
Note that the mapping $\mu \mapsto \eta_\mu$ gives a bijection
\begin{equation}\label{eq:homeo}
\{\text{probability measures on $\T$}\} \to \{f\colon \D\to \D\mid \text{holomorphic,  $f(0)=0$}\},  
\end{equation}
which was proved in \cite[Proposition 3.2]{BB05}. It is furthermore a homeomorphism with respect to weak convergence and locally uniform convergence \cite[Lemma 2.11]{FHS18}. 

It is clear from \eqref{eq:monotone_conv} and \eqref{eq:homeo} that the set of continuous $\circlearrowright$-CHs is in bijection with the set of multiplicative Loewner chains. 
To make it more explicit, a continuous $\circlearrowright$-CH $(\nu_{s,t})_{0 \le s \le t}$ (see Section \ref{sec:additive} for the definition) associates the family of analytic self-mappings $\eta_{s,t}=\eta_{\nu_{s,t}}$ of $\D$ which fixes $0$ and satisfies $\eta_{s,u} = \eta_{s,t}\circ \eta_{t,u}$ and $\eta_{s,s}= \id$ for all $0 \le s \le t \le u$, and the mapping $(s,t) \mapsto \eta_{s,t}$ is continuous with respect to locally uniform convergence, and thus $(\eta_{0,t})_{t\ge0}$ is a multiplicative Loewner chain.  Conversely a multiplicative Loewner chain associates the continuous $\circlearrowright$-CH as well.

\subsection{Ideas from generators of moments}\label{sec:monotone}
This section exhibits ideas for defining a bijection between continuous $\circledast$-CHs and $\circlearrowright$-CHs omitting detailed mathematical proofs. The point is to identify derivatives of moments of two kinds of CHs. 

Suppose that $(\mu_{s,t})_{t\ge s \ge 0}$ is a continuous $\circledast$-CH such that its generating family is smooth enough; for example 
\begin{equation}\label{eq:gen}
\alpha_t = \int_0^t \gamma_s \,ds \quad \text{and}\quad \sigma_t(B)= \int_0^t \rho_s(B)\,ds,  
\end{equation}
where $t\mapsto\gamma_t$ is continuous on $[0,\infty)$, and $\rho_t$ is a non-negative finite measure for every $t\ge0$ such that $t\mapsto \rho_t(B)$ is continuous for every $B \in \cB(\T)$. By \eqref{eq:increments_c} we deduce the time derivative of moments
\begin{equation}\label{eq:c_moments1}
\left. \frac{\partial}{\partial t}\right|_{t=s} \int_\T \xi^n \,d\mu_{s,t}(\xi)  = i \gamma_s n + \int_\T \frac{\xi^n-1- in  \Im(\xi)}{1-\Re(\xi)}\,\rho_s(d\xi), \qquad n \in \Z,~s \ge0. 
\end{equation}

On the other hand, suppose that $(\nu_{s,t})_{t\ge s \ge 0}$ is a continuous $\circlearrowright$-CH whose transition mappings $\eta_{s,t}:=\eta_{\nu_{s,t}}$ satisfy the Loewner differential equation from \cite[Proposition 3.12]{FHS18},  
\begin{equation}\label{eq:LDE}
\partial_t \eta_{s,t}(z) = - z p(z,t)\partial_z \eta_{s,t}(z), \qquad t \ge s,   
\end{equation}
where $p$ is of the form 
\begin{equation}\label{eq:MHVF}
p(z,t) = - i \gamma_t + \int_{\T} \frac{1+z\xi}{1-z\xi} \,d \rho_t(\xi).   
\end{equation}
Recall that $-z p(z,t)$ is called the (multiplicative) Herglotz vector field (of any order) for the Loewner chain $(\eta_{t})_{t\ge0}$ in \cite{CDMG14} or \cite{FHS18}. 
 Setting $t=s$ in \eqref{eq:LDE} we have $\partial_t \eta_{s,t}(z)|_{t=s} = - z p(z,s).$ Since $\psi_{\nu_{s,t}}(z) = \eta_{s,t}(z)/(1-\eta_{s,t}(z))$ we obtain 
\[
\partial_t \psi_{\nu_{s,t}}(z)|_{t=s} = - \frac{z}{(1-z)^2} p(z,s) = \frac{z}{(1-z)^2}\left[ i \gamma_s - \int_{\T} \frac{1+z\xi}{1-z\xi} \,d \rho_s(\xi)  \right]. 
\]
Using the identity
\begin{align*}
-\frac{z}{(1-z)^2} \cdot\frac{1+z\xi}{1-z\xi} 
&= \frac{1}{1-\Re(\xi)} \left[-  \frac{i \Im(\xi)z}{(1-z)^2} - \frac{z(1-\xi)}{(1-z)(1-z\xi)}\right] \\
&=   \sum_{n=1}^\infty z^n \frac{ \xi^n - 1 - in \Im(\xi)}{1-\Re(\xi)}, 
\end{align*}
and exchanging the sum and integral using the inequality \cite[Lemma 5.2]{HH}
\[
\left|\frac{\xi^n-1-i n\Im(\xi)}{1-\Re (\xi)}\right| \le n^3, \qquad \xi \in \T, ~n \in \N,  
\]
 we get 
\begin{equation}\label{eq:p1}
\partial_t \psi_{\nu_{s,t}}(z)|_{t=s} = \sum_{n\ge1} z^n\left[ i \gamma_s n +\int_{\T} \frac{\xi^n - 1 - i n \Im(\xi)}{1-\Re(\xi)} \, d \rho_s(\xi)\right], \qquad z \in \D.
\end{equation}
On the other hand, we can verify that 
\begin{equation}\label{eq:psi1}
\partial_t \psi_{\nu_{s,t}}(z)|_{t=s} = \sum_{n\ge1} z^n \left. \frac{\partial}{\partial t}\right|_{t=s} \int_{\T} \xi^n \,d\nu_{s,t}(\xi);  
\end{equation}
namely the sum and derivative commute, using the fact that $t\mapsto\eta_{s,t}(z)$ and hence $t\mapsto\psi_{s,t}(z)$ is of $C^1$, Cauchy's integral formula for the derivatives $\partial_z^n\psi_{s,t}(0)$, and Lebesgue's dominated convergence. Comparing the coefficients of \eqref{eq:p1} and \eqref{eq:psi1} together with formula \eqref{eq:c_moments1} we obtain
\begin{equation}\label{eq:f_moments}
 \left. \frac{\partial}{\partial t}\right|_{t=s} \int_{\T} \xi^n \,d\nu_{s,t}(\xi)=  \left. \frac{\partial}{\partial t}\right|_{t=s} \int_\T \xi^n \,d\mu_{s,t}(\xi), \qquad n \in \N,  
\end{equation}
which actually holds for $n \in \Z$ by complex conjugate. 

It is natural from \eqref{eq:f_moments} to define a map, denoted $\Theta_M$, sending the above $\circledast$-CH $(\mu_{s,t})_{t\ge s \ge 0}$ to the $\circlearrowright$-CH $(\nu_{s,t})_{t\ge s \ge 0}$, identifying ``time-dependent generators'' of moments of two kinds of CHs. This correspondence is also natural in terms of a ``dynamical Bercovici-Pata bijection'' in non-commutative probability; see Section \ref{sec:c-f}. 

 All those arguments at least require the differentiability of the parameters $(\alpha_t,\sigma_t)$. The main goal of Section \ref{sec4} is to extend the map $\Theta_M$ to  the set of all continuous $\circledast$-CHs, where the parameters $(\alpha_t,\sigma_t)$ satisfy assumptions \ref{LK1} and \ref{LK2} only. The problem is how to extract a generating family from a given continuous $\circlearrowright$-CH or a multiplicative Loewner chain. This can be done by generalizing the Loewner differential equation to the ``Loewner integral equation'' developed in Section \ref{sec:LIE}.

\subsection{The radial Loewner integral equation}\label{sec:LIE}

The Loewner integral equation we will introduce is governed by a certain driving measure-valued holomorphic function in the following sense.  

\begin{definition} Let $D$ be a domain of the complex plane. 
A family $\{Q(z,\cdot)\}_{z \in D}$ of locally finite complex measures on $[0,\infty)$ is called a \textit{Herglotz family of measures} (\textit{H-family} in short) if the following two conditions are satisfied;
\begin{enumerate}[label=\rm(HF\arabic*)]
\item \label{HF1} for each $B \in\cB_{b}([0,\infty))$ the function $z\mapsto Q(z,B)$ is holomorphic on $D$,
\item \label{HF2}  $\Re[Q(z,B)]\ge0$ for all $z\in D$ and $B \in \cB_{b}([0,\infty))$.   
\end{enumerate} 
An H-family $\{Q(z,\cdot)\}_{z \in D}$ is said to be \textit{continuous} if $Q(z,\{t\})=0$ for all $t \in [0,\infty)$ and $z\in D$. 
\end{definition}

\begin{proposition} \label{bi-measure}
A continuous H-family $\{Q(z,\cdot)\}_{z \in \D}$ has the form
\begin{equation}\label{G-measure}
Q(z,B) = i \Phi(B) + \int_\T \frac{1+\xi z}{1-\xi z} \,\Pi(d\xi \times B),   \qquad z \in \D, ~B \in \cB_{b}([0,\infty)), 
\end{equation}
where $\Phi$ is a locally finite signed measure on $[0,\infty)$ and $\Pi$ is a locally finite non-negative measure on $\T \times [0,\infty)$ such that $\Pi(\T \times \{t\})=0$ for every $t\ge0$. The pair $(\Phi,\Pi)$ is unique. Conversely, for such $\Phi$ and $\Pi$, \eqref{G-measure} generates a continuous H-family $\{Q(z,\cdot)\}_{z \in \D}$. 
\end{proposition}
\begin{proof}  
The uniqueness of $\Phi$ is a simple consequence of the fact $\Phi(B)= \Im[Q(0,B)]$. For the uniqueness of $\Pi$, for each $B$ the inversion formula (see \cite[equation (8) in Theorem IV.14, p.145]{Tsuji:1975} or \cite[Lemma 2.8]{FHS18}) 
\[
\Pi(A\times B) = \int_A \Re[Q(rw^{-1},B)]\, dw   
\]
holds, where $dw$ is the normalized Haar measure on $\T$ and $A$ is any arc whose endpoints are continuity points of $\Pi(\cdot \times B)$. This determines the values of $\Pi$ at bounded Borel subsets of the product form and hence at every bounded Borel subset of $\T\times [0,\infty)$, e.g.\ by Dynkin's $\pi$-$\lambda$ theorem. 

For the existence, it suffices to fix $T > 0$ and focus on the interval $[0,T]$; once we have obtained a finite signed measure $\Phi_T$ on $\cB([0,T])$ and a finite non-negative measure $\Pi_T$ on $\cB(\T\times [0,T])$ satisfying \eqref{G-measure} for all $B \in \cB([0, T])$ for each $T > 0$, we can prolong them by setting $\Phi(B) := \lim_{T\to\infty} \Phi_T(B)$ for $B\in \cB_{b}([0,\infty))$ and $\Pi(C) := \lim_{T\to\infty} \Pi_T(C)$ for all $C \in \cB(\T\times [0,\infty))$, which make sense thanks to the uniqueness established above. 

Now we fix $T>0$ and seek for a finite signed measure $\Phi$ on $\cB([0,T])$ and a finite non-negative measure $\Pi$ on $\cB(\T\times [0,T])$ satisfying \eqref{G-measure} for all $B \in \cB([0, T])$. By the Herglotz representation (see \cite[Theorem 2.4]{Pom75} or \cite[Lemma 2.8]{FHS18}), for every $B \in \cB([0,T])$ there exists $\Phi(B) \in \R$ and a non-negative finite measure $\Pi_0(\cdot, B)$ such that 
\begin{equation}
Q(z,B) = i \Phi(B) + \int_\T \frac{1+\xi z}{1-\xi z} \,\Pi_0(d\xi, B),   \qquad z \in \D 
\end{equation}
holds. Since $\Phi(B) = \Im[Q(0,B)]$, $\Phi$ is a signed measure on $[0,T]$. Note that since $\Pi_0(\T,B)=\Re[Q(0,B)]$ it follows that $\Pi_0(\T,\cdot)$ is a finite measure on $[0,T]$ and $\Pi_0(\T, \{t\})=0$ for every $t\ge0$. We will show that $\Pi_0$ extends to a measure on the product space $\T \times [0,T]$. Let $\tilde Q(z,\cdot) = Q(z,\cdot) - i\Phi(\cdot)$.

\vspace{2mm}\noindent
{\bf Step 1:} we prove that $\Pi_0(\cdot,\cdot)$ is a $\sigma$-additive measure on the second component as well.  
 Fix disjoint Borel subsets $B_1, B_2, \dots$ of $[0,T]$ and define 
$$
\Pi_1(A) := \Pi_0(A, \cup_{n\ge1} B_n) \quad \text{and}\quad \Pi_2(A) = \sum_{n\ge1} \Pi_0(A, B_n),  \qquad A \in \cB(\T), 
$$
which are finite measures on $\T$ with the common total mass $\Re[Q(0, \cup_{n\ge1} B_n)]$. 
We have on one hand 
\begin{align}
\tilde Q(z,\cup_{n\ge1} B_n) 
&= \int_{\T} \frac{1+ \xi z}{1-\xi z} \Pi_0(d\xi, \cup_{n\ge1}B_n) = \int_{\T} \frac{1+ \xi z}{1-\xi z} \Pi_1(d\xi)
\end{align}
and on the other hand, by Lemma \ref{lem:limits} \ref{item:limits2}
\begin{align}
\sum_{n\ge1} \tilde Q(z,B_n) 
&=   \sum_{n\ge1}  \int_{\T} \frac{1+ \xi z}{1-\xi z} \Pi_0(d\xi, B_n) =    \int_{\T} \frac{1+ \xi z}{1-\xi z} \Pi_2(d\xi) 
\end{align}
because $\xi\mapsto(1+\xi z)/(1-\xi z)$ is bounded and so $\Pi_2$-integrable. 
Since $\tilde Q(z,\cup_{n\ge1} B_n) = \sum_{n\ge1} \tilde Q(z,B_n)$, and by the uniqueness of the Herglotz representation, we have $\Pi_1 = \Pi_2$. Therefore, $\Pi_0(\cdot,\cdot)$ is a $\sigma$-additive measure on each component. 

\vspace{2mm}\noindent
{\bf Step 2:} we extend $\Pi_0$ to a finite non-negative measure $\Pi$ on $\T \times [0,T]$. 
Consider the semiring $\mathcal{S}=\{A \times [s,t): A \in \cA, 0 \le s < t \le T \}$, where $\cA$ is the set of arcs of the form $\{e^{i \theta}:\theta \in [\alpha,\beta)\}$ with $0\leq \beta- \alpha \leq \pi$. Define $\Pi \colon \mathcal{S} \to [0,\infty]$ by 
$
\Pi(A \times [s,t)) = \Pi_0(A,[s,t)). 
$ 
We can show that $\Pi$ is finitely additive and countably sub-additive on $\mathcal{S}$ by arguments similar to the proof of \cite[Theorem 12.5]{Bil2012}, and hence by \cite[Theorem 11.3]{Bil2012} it extends to a $\sigma$-finite $\sigma$-additive measure on the $\sigma$-field generated by $\mathcal{S}$, which is $\cB(\T \times [0,T])$. 

As an alternative proof, we can more directly resort to the characterization of two-dimensional distribution functions \cite[Theorem 12.5]{Bil2012}. To begin, we  identify $\T$ with $[0,2\pi)$ by the obvious bijection $\varphi\colon \{e^{i\theta}: \theta \in [0,2\pi)\} \to [0,2\pi)$ which is Borel measurable as well as its inverse. Define $\Xi_0\colon \cB([0,2\pi)) \times \cB([0,T])\to [0,\infty)$ by $\Xi_0(A,B) = \Pi_0(\varphi^{-1}(A),B)$ and the function $F(\theta,t)= \Xi_0((-\infty,\theta]\cap[0,2\pi) ,  (-\infty,t]\cap[0,T])$ on $\R^2$. The function $F$ satisfies the assumption of \cite[Theorem 12.5]{Bil2012}: it is continuous from the above thanks to the $\sigma$-additivity on each component and 
\begin{equation}\label{eqF}
F(\theta_2,t_2)- F(\theta_2,t_1) - F(\theta_1, t_2)+ F(\theta_1,t_1) \ge0
\end{equation}
 whenever $\theta_1 < \theta_2$ and $t_1< t_2$. In fact the LHS of \eqref{eqF} is exactly $\Xi_0((\theta_1,\theta_2]\cap[0,2\pi), (t_1,t_2]\cap[0,T])$ thanks to the $\sigma$-additivity for each variable of $\Xi_0$. 
Therefore, there exists a finite non-negative measure $\Xi$ on $\R^2$ such that $\Xi((-\infty,\theta] \times (-\infty, t]) = F(\theta,t)$. For $\theta \in [0,2\pi)$ and $t \in [0,T]$ we have $\Xi([0,\theta] \times [0, t])  = \Xi_0([0,\theta], [0, t]) $ as desired. 
\end{proof}

Now we are in a position to prove that a multiplicative Loewner chain with positive derivatives at $z=0$ satisfies an integral equation, called the \emph{Loewner integral equation} in this paper. 

\begin{theorem}\label{thm:LIE}
Let $(g_{t})_{t\ge 0}$ be a multiplicative Loewner chain such that $g_t'(0)>0$ for all $t\ge0$. 
Then there exists a unique continuous H-family $\{Q(z,\cdot)\}_{z \in \D}$ such that 
\begin{equation}\label{eq:LIE}
g_t(z) = z -z \int_0^t \frac{\partial g_s}{\partial z}(z) Q(z,ds),\qquad t \ge0, z \in \D. 
\end{equation}
Moreover, it holds that $\Im[Q(0, \cdot)]=0$, meaning that $\Phi$ in \eqref{G-measure} is zero. 
\end{theorem}
\begin{remark} 
If we drop the assumption $g_t'(0)>0$ then such an H-family that \eqref{eq:LIE} holds may not exist; see Section \ref{sec:gen}. 
\end{remark}
\begin{proof} 

It suffices to prove the statements on a fixed interval $[0,T]$. As in the proof of \cite[Theorem 3.16]{FHS18},  $g_t'(0)$ is of the form $e^{-\beta(t)}$, where $\beta(t)$ is a non-decreasing non-negative continuous function of $t$. We define $U=\beta(T)$ and 
\begin{equation}
\label{reparametrization}
\tau(u) = \sup \{ t \ge 0: \beta(t)=u \}, \qquad u \in [0,U]
\end{equation}
which is a strictly increasing function. Then $g_{\tau(u)}(z) = e^{-u}z + \cdots$ is a decreasing radial Loewner chain in the classical sense, in particular an absolutely continuous function regarding $u$.  According to \cite[Proposition 3.15]{FHS18}, there exists a multiplicative Herglotz vector field of the form $-z r (z,t)$ where $r$ is a function such that $\Re[r]\ge0$ and $r(0,t) =1$ for a.e.\ $t\ge0$, 
 such that 
$$
g_{\tau(u)}(z) = z -z \int_0^u \frac{\partial g_{\tau(v)}}{\partial z}(z) r(z,v) \,dv. 
$$
 Then we can define 
 \begin{equation}\label{eq:Q}
  Q(z,B)= \int_{\tau^{-1}(B)} r(z, v)\,dv 
\end{equation}
 for Borel subsets $B \subset[0,T]$. It satisfies $\Im[Q(0,\cdot)]=0$ and $Q(z,\{t\})=0$ for all $z \in \D$ and $t\ge0$. 
  By change of variables we have 
\begin{equation}\label{eq:change_of_variables}
\int_0^{T} k(t) Q(z,dt) =   \int_0^U k(\tau(v)) r(z,v) \,dv 
\end{equation}
for all bounded measurable functions $k$. By taking $k(t)= \partial_z g_t(z) 1_{[0,\tau(u)]}(t)$, we have
\begin{equation}\label{eq:integral}
g_{\tau(u)}(z) = z -z \int_0^{\tau(u)} \frac{\partial g_{t}}{\partial z} Q(z,dt), \qquad u \in [0,U]. 
\end{equation}

Here note that $[0,T]\setminus \tau([0,U])$ is a union of disjoint intervals of the form $[a,b)$ on which $\beta(t)$ takes constant values. On each such interval, the function $t \mapsto g_t(z)$ is constant by the Schwarz lemma. Also, $Q(z, \cdot)$ is the zero measure on $[0,T]\setminus \tau([0,U])$. 
Combining these facts together implies that equation \eqref{eq:integral} extends to 
$$
g_t(z) = z -z \int_0^t \frac{\partial g_s}{\partial z} Q(z,ds), \qquad t\in [0,T]. 
$$

It remains to show the uniqueness of $Q$.
Suppose that there exists another continuous H-family $Q^{*}$ such that $(g_t)_{t\ge0}$ satisfies \eqref{eq:LIE} with $Q^{*}$. Setting $L= -z(Q-Q^*)$ one obtains 
\begin{equation}
\label{uniqueness01}
\int_B  \frac{\partial g_s}{\partial z}(z) L(z, ds) =0
\end{equation}
for all $z \in \D$ and $B \in \cB([0,T])$. 
We fix $z\in \D$ and denote $L(z,dt)$ simply by $L(dt)$. According to \cite[Theorem 6.12]{Ru87} there exists a measurable function $h\colon [0,T] \to \C$ such that $|h|=1$ and $L(dt) = h(t)|L|(dt)$. The equation \eqref{uniqueness01} reads 
\begin{equation}\label{eq:integral_zero}
\int_B  \frac{\partial g_s}{\partial z}(z)h(s) |L|(ds) =0, \qquad B \in \cB([0,T]). 
\end{equation}
Since $g_s$ is univalent, the function $\ell(s):= \frac{\partial g_s}{\partial z}(z)h(s)$ does not vanish at any $s \in [0,T]$. This implies that $[0,T]=\{\Re[\ell(s)]>0\} \cup \{\Re[\ell(s)]<0\} \cup \{\Im[\ell(s)]>0\}\cup \{\Im[\ell(s)]<0\}$. 
Taking the real part of \eqref{eq:integral_zero} and choosing $B =\{\Re[\ell(s)]>0\}$ yields 
\[
\int_{\{\Re[\ell(s)]>0\}} \Re[\ell(s)] |L|(ds) =0, 
\]
which entails $|L|=0$ on the set $\{\Re[\ell(s)]>0\}$. Similar arguments apply to the other three sets and show that $|L|=0$ on $[0,T]$. 
\end{proof}

We establish a converse statement saying that the Loewner integral equation has a unique solution. 

\begin{theorem} \label{thm:solution_to_LIE}
Let $\{Q(z,\cdot)\}_{z \in \mathbb D}$ be a continuous H-family on $[0,\infty)$ with $\Im[Q(0,\cdot)]=0$. Then there exists a unique multiplicative Loewner chain $(g_{t})_{t\ge0}$ on $\D$ such that $g_t'(0)>0$ for all $t\ge0$ and \eqref{eq:LIE} is satisfied. 
\end{theorem}

\begin{proof}
Again we may focus on each finite interval $[0,T]$. For each Borel subset $B\in \cB([0,T])$ the function $Q(z,B)$ is of the form \eqref{G-measure}. Because of the assumption $\Im[Q(0,\cdot)]=0$ we must have $\Phi=0$. 

\vspace{2mm}\noindent
{\bf Step 1:} reduction to Herglotz vector fields of order 1. The idea is similar to the proof of Theorem \ref{thm:LIE}. Let 
\begin{equation}\label{eq:time_change}
\tau(u):= \sup\{t \ge0: \Pi(\T \times [0,t]) = u\}, \qquad R(z, B):=Q(z,\tau(B)).  
\end{equation}
The function $\tau$ is strictly increasing and hence $\Pi(A \times\tau(\cdot))$ is a finite measure for every $A \in \mathcal B(\T)$, and satisfies $\Pi(\T\times\tau(\cdot))=m(\cdot)$, where $m$ is the Lebesgue measure. Combined with \eqref{G-measure} those observations lead to
\begin{equation}\label{eq:estimate}
|R(z,B)| \leq \frac{2}{d(z,\T)} m(B) \leq \frac{2}{1-|z|} m(B)
\end{equation}
for every $z \in \D$ and $B \in \cB([0,T])$, where $d(\cdot,\cdot)$ is the standard distance on $\C$. This estimate implies that $R(z, \cdot)$ is absolutely continuous with respect to the Lebesgue measure $m$ for each fixed $z \in \D$, and then $\{n R(z,[t,t+\frac1n])\}_{n=1}^\infty$ defines a sequence of holomorphic functions on $\D$ converging to a function pointwisely on $\D$ as $n \to \infty$ for almost all $t \ge 0$, where the null set depends on $z$. 
The inequality \eqref{eq:estimate} shows that $\{n R(z,[t,t+\frac1n])\}_{n=1}^\infty$ is locally uniformly bounded on $\D$, and hence forms a normal family (Montel's theorem).

Now fix $k \in \{j \in \N : j \ge 2\}$ and let $z=1/k$.
Then $n R(\frac1k,[t,t+\frac1n])$ converges to a limit function, say $r(\frac1k, t)$, for all $t \in [0,T]\setminus E_{k}$, where $E_{k}$ is a null set.
It then follows that a limit $r(\frac1k, t)$ exists for all $k \in  \{j \in \N : j \ge 2\}$ and all $t \in [0,T]\setminus E$, where
\[
E := \bigcup_{k\ge 2} E_{k}. 
\]
Applying Vitali's theorem (e.g.\ \cite[p.150]{Rem98}), we conclude that there exists a function $r(z,t)$ such that for every $t \in [0,T]\setminus E$ the sequence $\{n R(\cdot,[t,t+\frac1n])\}_{n\ge1}$ converges to $r(\cdot,t)$ locally uniformly on $\D$ as $n\to\infty$. In particular, $r(\cdot,t)$ is holomorphic for every $t \in [0,T]\setminus E$, $\Re\, r\ge0$,  the function $t\mapsto r(z,t)$ is integrable for every $z \in \D$, and 
$$
R(z, B) = \int_B r(z,t)\,dt, \qquad z \in \D, ~B \in \cB([0,T]).  
$$
Since $R(0,B) = m(B)$, we have $r(0,t)=1$ for a.e.\ $t$. By \cite[Proposition 3.15]{FHS18}, the Loewner differential equation 
\begin{equation}\label{eq:LDE2}
\partial_t h(z,t) = -z r(z,t) \partial_z h(z,t), \qquad h(z,0)=z
\end{equation}
has a solution $h$ such that $h_t'(0)=e^{-t}$. 

\vspace{2mm}\noindent
{\bf Step 2:} going back to the original integral equation. The above equation \eqref{eq:LDE2} reads 
\begin{equation*}
h(z,u)  = z -z \int_0^u r(z,v) \partial_z h(z, v) \,dv = z -z \int_0^u \partial_z h(z, v) R(z,dv). 
\end{equation*}
Define 
\begin{equation*}
\label{def-f(z,t)}
g(z,t):= h(z,\tau^{-1}(t')), 
\end{equation*}
where 
\begin{equation*}
t'= \inf \{s \ge t: s \in \tau([0,T]) \} =\min \{s \ge t: s \in \tau([0,T]) \}. 
\end{equation*}
This definition obviously implies $g(z, \tau(u)) = h(z,u)$ and hence by the change of variables 
\begin{align*}
g(z,t) 
&=  z -z \int_0^{\tau^{-1}(t')} \partial_z g(z, \tau(u)) R(z,du) =   z -z \int_0^{t'}  \partial_z g(z, s) Q(z,ds). 
\end{align*}
By the definition of $\tau$, we know that $\Pi(\T\times[t,t'])=0$, and hence 
\begin{equation*}\label{eq:int}
g(z,t)  =  z -z \int_0^t \partial_z g(z, r) Q(z,dr),  
\end{equation*}
which verifies the existence of a solution. By definition, we have $g_t'(0)>0$. 

\vspace{2mm}\noindent
{\bf Step 3:} uniqueness. Suppose that $(\tilde g_t)_{t\ge0}$ is a solution. Then both $g_{\tau(u)}$ and $\tilde g_{\tau(u)}$ satisfy the same integral equation 
\eqref{eq:integral}. By the definition of $\tau$, the measure $\Pi(\T\times \cdot)$ and hence $Q(z,\cdot)$ vanishes outside the range of $\tau$ (see also the paragraph after \eqref{eq:integral}). This enables us to revert the definition of $R$ in \eqref{eq:time_change} into $Q(z, B) = R(z,\tau^{-1}(B)), B \in \cB_b([0,\infty))$ or, equivalently, \eqref{eq:Q}.  By the change of variables \eqref{eq:change_of_variables} applied to \eqref{eq:integral}, we conclude that both $g_{\tau(u)}$ and $\tilde g_{\tau(u)}$ satisfy the same differential equation \eqref{eq:LDE2}. 
Hence $g_{\tau(u)}=\tilde g_{\tau(u)}$ for all $u \ge0$ (see \cite[Theorem 4.1 and Theorem 4.2]{CDMG14}). From the paragraph after \eqref{eq:integral}, the equality also holds outside the range or $\tau$, so that $g_t=\tilde g_t$ for all $t\ge0$. 
\end{proof}

\begin{remark}
Yanagihara essentially proved Theorems \ref{thm:LIE} and \ref{thm:solution_to_LIE} in a different form \cite[Theorems 3.4 and 4.4]{Yan}. The relationship is as follows.  Let $(g_t)_{t\ge0}$ be a multiplicative Loewner chain such that $g_t'(0)>0$ for all $t\ge0$, $\{Q(z,\cdot)\}_{z\in \D}$ be its H-family and $\Pi$ be the measure as described in Proposition \ref{bi-measure} (see also Definition \ref{def:generating}). By the disintegration theorem (or the existence of regular conditional probability \cite[Section 10.2]{Dud02}), $\Pi$ can be decomposed as 
\[
\Pi(d\xi \times dt) = \tau_t(d\xi)da(t),
\]
where $a(t) = \Pi(\T\times[0,t])$ and $\tau_t$ is a probability measure for all $t\ge0$ such that $t\mapsto \tau_t(A)$ is measurable for every $A \in \cB(\T)$. Then 
\[
Q(z,dt) = \left[\int_{\T} \frac{1+ \xi z}{1- \xi z} \tau_t(d\xi)\right] da(t) =: b(z,t)da(t).  
\]
The Loewner integral equation for $(g_t)_{t\ge0}$ then reads  
\begin{equation}\label{eq:Yan}
g_t(z) = z - z \int_0^\infty \partial_z g(z,s) Q(z,ds) = z - z \int_0^\infty \partial_z g(z,s) b(z,s)\,da(s), 
\end{equation}
so that $b(z,t)$ is basically the function $P(z,t)/a(t)$ in \cite[Theorems 3.4 and 4.4]{Yan}; note that the Loewner chains in \cite{Yan} are increasing and  hence the sign in the Loewner equation is different from \eqref{eq:Yan}.

\end{remark}

\subsection{Generating families of multiplicative Loewner chains} \label{sec:gen}

In this section we make the ideas in Section \ref{sec:monotone} rigorous: we extract a generating family $(\alpha_t,\sigma_t)_{t\ge0}$ satisfying \ref{LK1} and \ref{LK2} from a given continuous $\circlearrowright$-CH.

For each continuous $\circlearrowright$-CH $(\nu_{s,t})_{t\ge s \ge 0}$ define $f_t= \eta_{\nu_{0,t}}$ which form a multiplicative Loewner chain, namely a decreasing Loewner chain of analytic self-mappings of $\D$ such that $t\mapsto f_t(z)$ is continuous for every $z \in \D$, $f_0=\id$ and $f_t(0)=0$. 

\vspace{2mm}
\noindent
{\bf \underline{Idea.}}  To get an idea for how to extract a generator $(\alpha_t,\sigma_t)_{t\ge0}$ from $(f_t)_{t\ge0}$, suppose first that $(f_t)_{t\ge0}$ is a multiplicative Loewner chain which is smooth enough (say, of order 1 in the sense of \cite[Definition 1.6]{CDMG14}, or $C^1$ as in Section \ref{sec:monotone}). According to \cite[Proposition 3.13]{FHS18},  $(f_t)_{t\ge0}$ satisfies the differential equation 
\begin{equation}\label{eq:p}
\partial_t f(z,t) =- z p(z,t) \partial_z f (z,t), 
\end{equation}
 where $p$ is a function of the form \eqref{eq:MHVF}. The idea in Section \ref{sec:monotone} was to define \eqref{eq:gen} to be the generator of $(f_t)_{t\ge0}$, but we met a difficulty that the arguments there required the differentiability of the parameters $\alpha_t,\sigma_t$. 

In order to extract a generator, we take the following detour. For a multiplicative Loewner chain $(f_t)_{t\ge0}$ which is just continuous on $t$, write $f_t'(0)=|f_t'(0)| e^{i\alpha_t}$ with a continuous function $\alpha_t \in \R$ such that $\alpha_0=0$ and define a supplementary function $g(z,t)=g_t(z) = f_t(e^{-i\alpha_t}z)$. The point is that $(g_t)_{t\ge0}$ is also a multiplicative Loewner chain with additional property $g_t'(0)>0$, so that Theorem \ref{thm:LIE} is available; denote by $\{Q(z,\cdot)\}_{z \in \D}$  the continuous H-family associated with $(g_t)_{t\ge0}$. We will relate this $Q$ with $\sigma_t$. For this purpose, again assume the smoothness of Loewner chains $(f_t)_{t\ge0}$ and hence of $(g_t)_{t\ge0}$ (later the smoothness will be removed in Definition \ref{def:generating}). Then the Loewner differential equation 
\begin{equation}\label{eq:abs_cont1}
\partial_t g(z,t) =- z q(z,t) \partial_z g (z,t) 
\end{equation}
holds for a function $q$ such that $q(\cdot,t)$ is analytic, $\Re[q(z,t)]\ge0$, and also $q(0,t) \in\R$ (thanks to $g_t'(0)>0$ and $g_t'(0)=\exp \left[-\int_0^t q(0,s)\,ds\right]$), so that it has the Herglotz representation 
\begin{equation*}
q(z,t) = \int_{\T} \frac{1+\xi z}{1-\xi z} d\tau_t(\xi)
\end{equation*}
for some finite non-negative measure $\tau_t$ such that $t\mapsto \tau_t(\T)$ is locally integrable. Note that $Q(z,dt)=q(z,t)\,dt$, and so 
\begin{equation}\label{eq:abs_cont3}
Q(z,B) = \int_B q(z,t)\,dt = \int_\T \frac{1+\xi z}{1-\xi z} \, \Pi(d\xi \times B), \qquad B\in \cB_b([0,\infty)),
\end{equation}
where $\Pi$ is a locally finite non-negative measure on $ \T\times [0,\infty)$ defined by $\Pi(d\xi \times dt)=\tau_t(d\xi)dt$. 

On the other hand, the relationship between $p$ and $q$ is given by 
\[
p(z,t)=q(e^{i\alpha_t} z,t)- i\dot \alpha_t, 
\]
because 
\begin{align*}
\partial_t f (z,t) 
&= \partial_t  g (e^{i\alpha_t} z,t)+ i\dot\alpha_t e^{i\alpha_t} z \partial_z g(e^{i\alpha_t} z, t) \\
&= z (- e^{i\alpha_t}q(e^{i\alpha_t} z,t) + i\dot\alpha_te^{i\alpha_t})  \partial_z g(e^{i\alpha_t} z, t) \\
&=  -z (q(e^{i\alpha_t} z,t)- i\dot\alpha_t)  \partial_z f(z,t). 
\end{align*}
So the Herglotz function $p(\cdot,t)$ associated with $(f_{t})$ is 
\begin{equation*}
p(z,t)=q(e^{i\alpha_t} z,t)- i\dot \alpha_t = - i\dot \alpha_t +  \int_{\T} \frac{1+ e^{i\alpha_t}\xi z}{1- e^{i\alpha_t} \xi z} d\tau_t(\xi). 
\end{equation*}
and hence the continuous H-family for $(f_{t})_{t\ge0}$ is given by 
\begin{equation}\label{eq:abs_cont2}
P(z,B) := \int_B p(z,t) \, dt = -i \int_B \dot \alpha_t\, dt +  \int_{\T} \frac{1+\xi z}{1- \xi z} \Sigma(d\xi \times B), 
\end{equation}
where $\Sigma$ is the push-forward of $\Pi$ by the map $(\xi,t)\mapsto (e^{i\alpha_t}\xi,t)$ on $ \T\times [0,\infty)$. Finally, the measure $\sigma_t$ is defined by $\sigma_t(A) = \Sigma(A \times [0,t])$. 

\vspace{2mm}\noindent{\bf \underline{General case.}} For a general multiplicative Loewner chain $(f_{t})_{t\ge0}$ the function $\alpha_t$ is only continuous and it is difficult to give a reasonable definition of $\dot\alpha_t\, dt$ and hence of $P(z,dt)$. However, in view of the discussions above, $\alpha_t$ and $\sigma_t$ can still be defined as follows.  

\begin{definition}\label{def:generating}  Let $(f_{t})_{t\ge0}$ be a multiplicative Loewner chain. 
\begin{enumerate}[label=\rm(\roman*)]
\item Define a continuous function $t\mapsto \alpha_t$ such that $f_t'(0)=|f_t'(0)| e^{i\alpha_t}$ and $\alpha_0=0$; 

\item apply Theorem \ref{thm:LIE} to the multiplicative Loewner chain $g_t(z)=  f_t(e^{-i\alpha_t}z)$ and take the continuous H-family $\{Q(z,dt)\}_{z \in \D}$ for $(g_t)_{t\ge0}$ with the representation 
\begin{equation*}
Q(z,B) = \int_{\T} \frac{1+ \xi z}{1- \xi z} \Pi(d\xi \times B),  \qquad z \in \D,~B \in \cB_{b}([0,\infty));
\end{equation*}

\item define $\Sigma$ to be the push-forward of $\Pi$ by the mapping $(\xi,t)\mapsto (e^{i\alpha_t}\xi,t)$ and also $\sigma_t(\cdot)= \Sigma(\cdot,[0,t])$.  

\end{enumerate}
\end{definition}

Again, the point is that we cannot define an H-family for $(f_t)_{t\ge0}$ because of the high singularity of $\dot \alpha_t \,dt$, but we can still define a pair $(\alpha_t, \sigma_t)_{t\ge0}$.

This definition is compatible with the smooth case discussed earlier: the functions $p, q$ in \eqref{eq:p} and \eqref{eq:abs_cont1} are then related to $\Sigma,\Pi$ via \eqref{eq:abs_cont2} and \eqref{eq:abs_cont3}. 

Definition \ref{def:generating} gives a mapping from the set of the multiplicative Loewner chains to the set of the families $(\alpha_t,\sigma_t)_{t\ge0}$ satisfying \ref{LK1} and \ref{LK2}. This mapping is well defined by the unique existence of the H-family in Theorem \ref{thm:LIE}, it is surjective by the existence of the solution in Theorem \ref{thm:solution_to_LIE}, and injective by the uniqueness of the solution in Theorem \ref{thm:solution_to_LIE}. Furthermore it is homeomorphic with respect to suitable topologies, see Theorem \ref{thm:conv_monotone} below. 

In \cite{FHS18} a bijection is established between the set of multiplicative Loewner chains and the set of continuous $\circlearrowright$-CHs that arise from monotone probability. We call the associated $(\alpha_t,\sigma_t)_{t\ge0}$ the \emph{generating family} for the corresponding continuous $\circlearrowright$-CH. 

\begin{definition}\label{def:c-f} A bijection 
\[\Theta_M\colon\{\text{continuous $\circledast$-CHs on $\T$}\} \to \{\text{continuous $\circlearrowright$-CHs on $\T$}\}
\]
is defined by sending the continuous $\circledast$-CH characterized by a generating family to the continuous $\circlearrowright$-CH having the same generating family. 
\end{definition}

Note that a continuous $\circledast$-CH is time-homogeneous if and only if $(\alpha_t,\sigma_t)_{t\ge0}$ is of the form $(t\alpha, t\sigma)_{t\ge0}$ for some $\alpha \in \R$ and a finite non-negative measure $\sigma$ if and only if the corresponding multiplicative Loewner chain $(f_t)_{t\ge0}$ is a semigroup, namely $f_t \circ f_s = f_{t+s}$ for all $s,t\ge0$.

\begin{example} The normal distribution on the unit circle $N_\T(m,v)$ is defined to be the law of $e^{i Z}$ where $Z$ is distributed as $N(m,v)$. 
The $\circledast$-semigroup of normal distributions on the unit circle $(N_\T(0,t))_{t\ge0}$ is characterized by $(\alpha_t,\sigma_t)=(0,(t/2)\delta_1)$, and is mapped by $\Theta_M$ to the $\circlearrowright$-semigroup consisting of the marginal distributions of monotone unitary multiplicative Brownian motion introduced in \cite{Ham15}. 

\end{example}

\begin{example} 
Let $f\colon [0,\infty) \to \T$ be a measurable function and $\lambda\colon [0,\infty)\to [0,\infty)$ be a locally integrable function. The slit multiplicative Loewner chain governed by the Herglotz vector field 
\[
p(z,t) = \lambda(t)\frac{1+z f(t)}{1-z f(t)}
\]
associates a continuous $\circlearrowright$-CH whose generating family $(\alpha_t,\sigma_t)$ is given by $\alpha_t=0$ and $\sigma_t(B)= \int_0^t\lambda(s)1_{B}(f(s))\,ds$. 
\end{example}

\subsection{Convergence of Loewner chains}\label{sec:conv}

Motivated by probability theory, we defined a time-dependent generator in Section \ref{sec:gen}. Now we show that this definition is appropriate by demonstrating that convergence of Loewner chains is equivalent to that of generating families. 

\begin{lemma}\label{lem:diff_eq}
Let $\rho$ be an atomless, locally finite non-negative measure on $[0,\infty)$ and $a,b\in\R$. Then the integral equation 
$$
u(t) = a + b\int_0^t u(s) \,d\rho(s), \qquad t \ge0, 
$$
has a unique solution $u \in C([0,\infty))$ given by 
\[
u(t) = a\exp(b\rho([0,t])). 
\]
\end{lemma}
\begin{proof} Let $T>0$ and $W$ be a mapping on $C[0,T]$ defined by
\[
W[u](t)= a+ b\int_0^t u(s) \,d\rho(s). 
\]
Then 
\[
\|W[u] - W[v]\|_{C[0,T]}  \le |b|\rho([0,T])\|u-v\|_{C[0,T]}, 
\]
and hence $W$ is a contraction mapping if $T>0$ is chosen such that $|b|\rho([0,T])<1$. This implies (the local existence and) uniqueness of a solution. 
  
Let $u_n$ be functions defined by 
\begin{equation}\label{eq:rec}
u_0(t)=a, \qquad u_n(t) = a + b\int_0^t u_{n-1}(s)\,d\rho(s), \qquad n \in \N. 
\end{equation}
We obtain 
\begin{align*}
u_{n}(t) - u_{n-1}(t)
&= b\int_0^t [u_{n-1}(s_1) - u_{n-2}(s_1)] \,d\rho(s_1) \\ 
& \cdots \\
&=  b^{n-1}\int_0^t d\rho(s_1) \int_0^{s_1} d\rho(s_2) \cdots  \int_0^{s_{n-2}} [u_1(s_{n-1})- u_0(s_{n-1})] d\rho(s_{n-1}) \\
&= a b^n\rho^{\otimes n} (\Delta_n(t)) = \frac{ab^n \rho^{\otimes n}([0,t]^n)}{n!} = \frac{a(b\rho([0,t]))^n}{n!}, 
\end{align*}
where $\Delta_{n}(t) = \{(s_1,s_2,\dots, s_n) \in [0,t]^n:  s_1\ge s_2 \ge \cdots \ge s_n \ge 0 \}$. This implies that
\[
u_n(t) = a\sum_{k=0}^n\frac{(b\rho([0,t]))^k}{k!} \to u_\infty(t):=a\exp[b\rho([0,t])]
\]
locally uniformly as $n\to\infty$. Taking the limit in \eqref{eq:rec} we conclude that $u_\infty$ is the solution. 
\end{proof}

Now we state the main convergence result. 

\begin{theorem}\label{thm:conv_monotone} Let $(f_{t})_{t\ge0}$, $(f_t^{(n)})_{t\ge0}$, $n=1,2,3,\dots,$ be multiplicative Loewner chains, $(f_{s,t})_{t\ge s \ge 0}$, $(f_{s,t}^{(n)})_{t\ge s \ge 0}$ be their reverse evolution families and $(\nu_{s,t})_{t\ge s \ge0}$, $(\nu_{s,t}^{(n)})_{t\ge s \ge0}$ be the associate continuous $\circlearrowright$-CHs with generating families $(\alpha_t,\sigma_t)_{t\ge0}, (\alpha_t^{(n)},\sigma_t^{(n)})_{t\ge0}$, respectively.  
 The following conditions are equivalent as $n\to \infty$:  
\begin{enumerate}[label=\rm(M\arabic*)]
\item\label{condM1} $(f_t^{(n)})_{t\ge0}$ converges to $(f_t)_{t\ge0}$ locally uniformly on $\D \times [0,\infty)$; 

\item\label{condM2} $(f_{s,t}^{(n)})_{t\ge s \ge 0}$ converges to $(f_{s,t})_{t\ge s \ge 0}$ locally uniformly on $\D \times \{(s,t): 0 \le s \le t <\infty\}$; 

\item\label{condM3} $(\nu_{0,t}^{(n)})_{t\ge 0}$ weakly converges to $(\nu_{0,t})_{t \ge 0}$ locally uniformly on $[0,\infty)$;

\item\label{condM4} $(\nu_{s,t}^{(n)})_{t\ge s \ge 0}$ weakly converges to $(\nu_{s,t})_{t\ge s \ge 0}$ locally uniformly on $\{(s,t): 0 \le s \le t <\infty\}$;  

\item \label{condM5}$(\alpha_t^{(n)})_{t\ge0}$ converges to $(\alpha_t)_{t\ge0}$ locally uniformly on $[0,\infty)$ and the mutually equivalent conditions \ref{condG1}--\ref{condG4} in Proposition \ref{prop:convergence_equivalence} hold. 
\end{enumerate} 
\end{theorem}
\begin{proof}  
Let $g_t(z) = f_t(e^{-i\alpha_t}z)$ and $g_t^{(n)}=f_t^{(n)}(e^{-i\alpha_t^{(n)}}z)$. 
We will proceed as \ref{condM1} $\Leftrightarrow$ \ref{condM2}, \ref{condM1} $\Leftrightarrow$ \ref{condM5}, \ref{condM2} $\Rightarrow$ \ref{condM4} $\Rightarrow$ \ref{condM3}  $\Rightarrow$ \ref{condM1}. Among them the implications \ref{condM2} $\Rightarrow$ \ref{condM1} and \ref{condM4} $\Rightarrow$ \ref{condM3} are obvious. 

\vspace{2mm}\noindent\underline{\ref{condM1} $\Rightarrow$ \ref{condM2}.}  Recall that $f_{s,t}^{(n)}:= (f_s^{(n)})^{-1} \circ f_t^{(n)}$ and $f_{s,t}:= f_s^{-1} \circ f_t$. Fix $T>0$ and $0<r_0<r<1$ and let $K= r_0 \overline{\D}$. Since $f_{s,t}^{(n)}(r \overline{\D}) \subset r \overline{\D}$ by the Schwarz lemma, we obtain
\begin{equation}\label{eq:Sch}
f_t^{(n)}(K) \subset\subset f_t^{(n)}(r \overline{\D}) =  f_s^{(n)}(f_{s,t}^{(n)}(r \overline{\D}))\subset f_s^{(n)}(r \overline{\D})
\end{equation}
for all $0 \le s \le t <\infty $ and $n\ge1$. Therefore, $f_t^{(n)}(K)$ is surrounded by the simple curve $f_s^{(n)}(r \T)$ and they do not have intersection, and hence the Lagrange inversion formula for $(f_s^{(n)})^{-1}$ implies 
\[
f_{s,t}^{(n)}(z)=(f_s^{(n)})^{-1}(f_t^{(n)}(z)) = \frac{1}{2\pi i} \int_{r\T} \frac{w (f_s^{(n)})'(w)}{f_s^{(n)}(w)-f_t^{(n)}(z)}\,dw, \qquad z \in K,   n \ge n_0. 
\]
Therefore, in order to prove the uniform convergence of $f_{s,t}^{(n)} \to f_{s,t}$ on $K\times\{(s,t): 0 \le s \le t \le T\}$, it suffices to prove that 
\begin{equation}\label{incl}
c:=\inf_{0 \le s \le t \le T, n\ge1}d(f_t^{(n)}(K), f_s^{(n)}(r\T))>0, 
\end{equation}
where $d(\cdot, \cdot)$ is the Euclidean distance of sets. 
Since the inclusion $f_t^{(n)}(K) \subset f_s^{(n)}(K)$ holds from the same arguments used in \eqref{eq:Sch}, we have actually 
\[
c=\inf_{0 \le s \le T, n\ge1}d(f_s^{(n)}(K), f_s^{(n)}(r\T)). 
\]
Now suppose to the contrary that $c=0$. We may find a sequence $\{s_k\}_{k=1}^\infty$ of $[0,T]$ and a sequence $\{n_k\}_{k\ge1}$ of $\N$ such that 
\begin{equation}\label{eq:hyp}
d(f_{s_k}^{(n_k)}(K), f_{s_k}^{(n_k)}(r\T)) \to 0.   
\end{equation}

{\bf Case (a):} if $\{n_k\}$ is unbounded, we may assume, passing to a subsequence if necessary, that $s_k$ converges to a certain $u$ and that $n_1 < n_2 < \cdots$.  Since $f_t^{(n)}$ converges to $f_t$ uniformly on $K \times [0,T]$, we have, for every $\epsilon>0$,  
\[
\exists N \in \N~ \forall n \ge N ~\forall s \in [0,T]:   f_s^{(n)}(K) \subset f_s(K)^\epsilon, f_s^{(n)}(r\T) \subset f_s(r\T)^\epsilon,  
\]  
where $A^\epsilon$ is the epsilon neighborhood $\{z \in \C: d(z,A)<\epsilon\}$ of a subset $A\subset \C$. Furthermore, we have $f_{s_k}^{(n_k)}(K) \subset f_{s_k}(K)^\epsilon \subset f_{u}(K)^{2\epsilon}$ and $f_{s_k}^{(n_k)}(r\T) \subset f_{s_k}(r\T)^\epsilon\subset f_{u}(r\T)^{2\epsilon}$ whenever $k$ is large enough, and hence,  by choosing $\epsilon= \frac{1}{5}d(f_{u}(K), f_u(r\T))>0$ we get
$$
d(f_{s_k}^{(n_k)}(K), f_{s_k}^{(n_k)}(r\T)) \ge d(f_{u}(K)^{2\epsilon}, f_{u}(r\T)^{2\epsilon}) \ge \epsilon, 
$$
 a contradiction to \eqref{eq:hyp}. 
 
{\bf Case (b):} if $\{n_k\}$ is bounded then we can assume that $n_1=n_2=\cdots $ by passing to a subsequence. The remaining arguments are simpler than case (a). 
 
\vspace{2mm}\noindent\underline{\ref{condM1} $\Rightarrow$ \ref{condM5}.} Fix $T>0$. By the assumption, $\exp[i\alpha_t^{(n)}]=(f_t^{(n)})'(0)/|(f_t^{(n)})'(0)|$ converges to $\exp[i\alpha_t]$ uniformly, and so does $\alpha_t^{(n)}$ to $\alpha_t$. This also implies that $g_t^{(n)}$ converges to $g_t$ locally uniformly on $\D \times [0,T]$. Let $\Pi$ and $\Pi^{(n)}$ be the measures associated to $g_t$ and $g_t^{(n)}$ in Definition \ref{def:generating}, respectively. The Loewner integral equation 
\begin{equation}\label{eq:int}
g_t^{(n)}(z) = z - z \int_{ \T\times [0,t]} \partial_z g_s^{(n)}(z) \frac{1+ \xi z}{1- \xi z} \Pi^{(n)}(d\xi \times ds), \qquad z\in \D, ~t \in [0,T]
\end{equation}
specializes to 
\[
\partial_z g_t^{(n)}(0) = 1- \int_0^t \partial_z g_s^{(n)}(0) \,\Pi^{(n)}(\T,ds). 
\]
Lemma \ref{lem:diff_eq} yields $\partial_z g_t^{(n)}(0) = \exp(- \Pi^{(n)}(\T\times[0,t]))$, and hence 
\begin{equation}\label{eq:loc_unif_sigman}
\Sigma^{(n)}(\T\times[0,t])=\Pi^{(n)}(\T\times[0,t])\to \Sigma(\T\times[0,t])
\end{equation}
 uniformly on $[0,T]$. In particular, $\Sigma^{(n)}|_{\T \times [0,T]}$ is uniformly bounded, and hence by Prohorov's theorem (see Theorem \ref{thm:Prohorov}), it has a subsequence, still denoted by $\Sigma^{(n)}|_{\T \times [0,T]}$, weakly converging to some finite non-negative measure on $\T \times [0,T]$, denoted by $\tilde\Sigma$. From \eqref{eq:loc_unif_sigman} we infer that the marginal measure $\Sigma^{(n)}(\T \times \cdot)$ weakly converges to both $\Sigma(\T\times \cdot)$ and $\tilde \Sigma(\T\times \cdot)$ on $[0,T]$, and hence those limits coincide. In particular, $t\mapsto \tilde \Sigma(\T\times [0,t])$ is continuous. The measure $\Pi^{(n)}|_{\T \times [0,T]}$ also converges weakly to the measure $\tilde\Pi$ corresponding to $\tilde \Sigma$ because the push-forward map $(\xi,t)\mapsto (\exp(i \alpha_t^{(n)})\xi,t)$ is uniformly convergent as $n\to\infty$. Passing to the limit in \eqref{eq:int} we obtain 
\[
g_t(z) = z - z \int_{ \T\times [0,t]} \partial_z g_s(z) \frac{1+ \xi z}{1- \xi z} \tilde\Pi(d\xi \times ds), \qquad z\in \D, ~t \in [0,T].
\]
By the uniqueness (Theorem \ref{thm:LIE}) we conclude that $\tilde \Pi = \Pi|_{\T \times [0,T]}$ and hence $\tilde\Sigma=\Sigma|_{\T \times [0,T]}$. These arguments verify condition \ref{condG4}. 

\vspace{2mm}\noindent\underline{\ref{condM5} $\Rightarrow$ \ref{condM1}.} 
By Theorem \ref{thm:estimate_LC}, we have 
\begin{align}
|g_t^{(n)}(z) - g_s^{(n)}(z)| 
&\le \frac{8|z|}{(1-|z|)^4} |e^{\Pi^{(n)}(\T\times[0,t])} - e^{\Pi^{(n)}(\T\times[0,s]))}|  \label{eq:est}
\end{align}
for $z \in \D$ and $t\ge s \ge 0$. Fix $T>0$, a compact subset $K$ of $\D$ and $r\in(0,1)$ such that $K \subset  r\D $. Let $d$ be the distance of $r\T$ and $K$.  Combining \eqref{eq:est} and Cauchy's integral formula yields that, for all $z,w \in K$ and $s,t \in [0,T]$, 
\begin{align}
|g_t^{(n)}(z) - g_s^{(n)}(w)|  
& \le |g_t^{(n)}(z) - g_t^{(n)}(w)| +|g_t^{(n)}(w) - g_s^{(n)}(w)|  \notag\\ 
&\le \int_{r\T} \frac{  |z-w| |g_t^{(n)}(\xi)|}{2\pi|(\xi -z)(\xi -w)|} |d\xi| + \int_{r\T} \frac{|g_t^{(n)}(\xi) - g_s^{(n)}(\xi)|}{2\pi|\xi -w|} |d\xi|  \notag \\
&\leq \frac{r}{d^2}|z-w| + \frac{8r^2 e^{\Pi^{(n)}(\T,[0,T])}}{d (1-r)^4} |1- e^{-\Pi^{(n)}(\T,(s,t])}|. \label{eq:est2}
\end{align}
 As in the proof of Proposition \ref{prop:convergence_equivalence} we have 
\begin{equation}\label{eq:equi}
\lim_{h\to0} \sup_{\substack{s,t \in [0,T], 0 \le t-s <h\\n \in \N}} |\Pi^{(n)}(\T,(s,t])| =0. 
\end{equation}
 Combining \eqref{eq:est2} and \eqref{eq:equi} implies that the family $\{g_t^{(n)}(z)\}_{n\ge1}$ is equi-continuous on $[0,T] \times K$. It is $\D$-valued and hence uniformly bounded.  By Arzel\`a-Ascoli's theorem $\{g_t^{(n)}(z)\}_{n\ge1}$ has a convergent subsequence in the locally uniform topology on $C([0,\infty) \times \D)$. Denote by $\tilde g_t$ its limit. In \eqref{eq:int} passing to the limit one arrives at 
\[
\tilde g_t(z) = z - z \int_{\T \times [0,t]} \partial_z \tilde g_s(z) \frac{1+ \xi z}{1- \xi z} \Pi(d\xi \times ds). 
\]
By the uniqueness of the solution we conclude that $\tilde g_t = g_t$. Therefore, the whole sequence $\{g_t^{(n)}\}_{n\ge 1}$ converges to $g_t$ locally uniformly on $[0,\infty) \times \D$ and so does $\{f_t^{(n)}\}_{n\ge 1}$ to $f_t$. 

\vspace{2mm}\noindent\underline{\ref{condM2} $\Rightarrow$ \ref{condM4}.}  Recall that $f_{s,t}^{(n)} =\eta_{\nu_{s,t}^{(n)}}$ and $f_{s,t} =\eta_{\nu_{s,t}}$.  The functions $\psi_{s,t}^{(n)}:= f_{s,t}^{(n)}/(1-f_{s,t}^{(n)})$ converge to $\psi_{s,t}:=f_{s,t}/(1-f_{s,t})$ locally uniformly on $s,t,z$ as well. Using the formula
\[
 \int_\T \xi^k \,d\nu_{s,t}^{(n)}(\xi) =  \frac{1}{2\pi i}\int_{|\xi|=1/2} \frac{\psi_{s,t}^{(n)}(\xi)}{\xi^{k+1}}\,d\xi\qquad \text{and}\qquad \int_\T \xi^k \,d\nu_{s,t}(\xi) =  \frac{1}{2\pi i}\int_{|\xi|=1/2} \frac{\psi_{s,t}(\xi)}{\xi^{k+1}}\,d\xi
\]
one obtains the desired conclusion for test functions $f(\xi)= \xi^k, k\in \N$. For general continuous functions one may use Stone-Weierstrass' theorem.

\vspace{2mm}\noindent\underline{\ref{condM3} $\Rightarrow$ \ref{condM1}.}  Recall that $f_{t}^{(n)} =\eta_{\nu_{0,t}^{(n)}}$ and $f_{t} =\eta_{\nu_{0,t}}$. Take $r\in (0,1)$ and $T,\epsilon>0$.  Pick $p\in \N$ such that $\sum_{k> p} r^k <\epsilon$. We have
\begin{align*}
|\psi_{0,t}^{(n)}(z) - \psi_{0,t}(z)| \le \sum_{k =1}^p \left| \int_\T \xi^k \,d\nu_{0,t}^{(n)}(\xi) -\int_\T \xi^k \,d\nu_{0,t}(\xi)\right| + 2\epsilon, \qquad t \ge0, z \in r\D. 
\end{align*}
Taking the supremum over $(z,t) \in r\D \times  [0,T]$ and letting $n\to\infty$ show the locally uniform convergence of $\psi_{0,t}^{(n)}(z)$ and hence of $f_t^{(n)}(z)$. 
\end{proof}
Combining Theorem \ref{thm:conv_classical} and Theorem \ref{thm:conv_monotone} establishes the following property of $\Theta_M$.

\begin{corollary} The bijection $\Theta_M$ is a homeomorphism with respect to the locally uniform weak convergence of convolution hemigroups. 
\end{corollary}

Note that the bijection $\Theta_M$ induces a bijection between classical and monotone continuous convolution semigroups. 
This is also homeomorphic. 

\begin{remark}
The conditions \ref{condM1}--\ref{condM5} are also equivalent to 
\begin{enumerate}[label=(M\arabic*')]
\item\label{condM1'} $\alpha_t^{(n)}$ converges to $\alpha_t$ locally uniformly on $[0,\infty)$, and $g_t^{(n)}$ converges  to $g_t$ locally uniformly on $\D$ for every $t\ge0$,  
\end{enumerate}
where $g_t, g_t^{(n)}$ are the Loewner chains introduced in the proof of Theorem \ref{thm:conv_monotone}. 
The proof of the implication \ref{condM1'} $\Rightarrow$ \ref{condM1} is sketched below. Since $\Pi(\T\times [0,t]) = -\log g_t'(0)$ and $\Pi^{(n)}(\T\times [0,t])$ are continuous and non-decreasing, by Polya's theorem \cite[Theorem 4.5]{BW16}, the pointwise convergence implies the locally uniform convergence. Then we can prove that the family $\{g_t^{(n)}\}_{n\ge1}$ is equi-continuous on each compact subset of $\D \times [0,\infty)$ by the estimates \eqref{eq:est2} and \eqref{eq:equi}.

We cannot replace \ref{condM1'} by ``$f_t^{(n)}$ converges  to $f_t$ locally uniformly on $\D$ for every $t\ge0$'', because $(f_t^{(n)})'(0)=\exp(i\alpha_t^{(n)})$ and $f_t'(0)=\exp(i\alpha_t)$ can be quite general continuous functions, so that just the pointwise convergence does not imply the locally uniform convergence in general.  
\end{remark}

\begin{remark}
Theorem \ref{thm:conv_monotone} has intersections with several known convergence results. Roth gave a sufficient condition for a locally uniform convergence of Loewner chains in \cite[Lemma I.37]{Rot98}, where the Denjoy-Wolff points of Loewner chains are arbitrary (see also \cite[Lemma 3.1]{Hot19}; Roth puts an additional parameter).  Johansson, Sola and Turner gave a sufficient condition for the convergence of Loewner chains at a fixed time in the setting that $\alpha_t=\alpha_t^{(n)}=0$ and that $\Sigma(\T\times \cdot), \Sigma^{(n)}(\T\times \cdot)$ are the Lebesgue measure \cite[Proposition 1]{JST12}. Miller and Sheffield proved the equivalence \ref{condM1'} $\Leftrightarrow$ \ref{condM5} when $\alpha_t=\alpha_t^{(n)}=0$ and $\Sigma(\T\times \cdot), \Sigma^{(n)}(\T\times \cdot)$ are the Lebesgue measure \cite[Theorem 1.1]{MS16}. 
\end{remark}

\section{Bijections between classical, free and boolean convolution hemigroups}\label{sec:c-f}
This section aims to add further two objects in Figure \ref{fig:star} in Section \ref{sec:main}: the set of continuous free CHs and the set of continuous boolean CHs. Also, an embedding of free CHs into monotone CHs, based on subordination, will be discussed. 

In this section we use the notion of infinite divisibility. Given an associative binary operation $\star$ on the set of probability measures on $\T$, a probability measure $\mu$ is said to be \emph{$\star$-infinitely divisible} (or $\star$-ID for short) if, for every $n\in\N$, there exists a probability measure $\mu_n$ such that $\mu= \mu_n\star \mu_n \star \cdots \star \mu_n$ ($n$-fold).

\subsection{Free convolution hemigroups on the unit circle}\label{sec:c-f2}

A $\circledast$-ID distribution without idempotent factors (or equivalently, with non-zero mean) has a L\'evy-Khintchine representation, but its L\'evy measure is not unique in general, see \cite[Remark 3, Chapter IV]{Par67}. Because of this, one is unable to define a bijection of Bercovici-Pata type from the set  of $\circledast$-ID distributions with non-zero mean to the set of $\boxtimes$-ID distributions with non-zero mean. For further discussions, see \cite{Ceb16,CG08b}. 
However, thanks to the uniqueness of the continuous family $(\alpha_t,\sigma_t)_{t\ge0}$ in Theorem \ref{thm:LK} we are able to define a dynamical version, that is, a bijection from continuous $\circledast$-CHs to continuous $\boxtimes$-CHs, see Definition \ref{def:c-f}. 

Here we introduce some notions in free probability; the reader is referred to \cite{BV92} for further details. For free unitary operators $U$ and $V$ whose distributions on $\T$ are $\mu$ and $\nu$ respectively, the distribution of $UV$ is called the \emph{free multiplicative convolution} of $\mu$ and $\nu$, denoted by $\mu \boxtimes \nu$. Free multiplicative convolution can be computed by the $\Sigma$-transform. 
If the mean $\int_{\T} \xi \,d\mu(\xi)$ of a probability measure $\mu$ is non-zero, then the inverse function $\eta_\mu^{-1}(z)$ can be defined in a neighborhood of 0 as a convergent power series, and the function 
\begin{equation}\label{eq:sigma-transform}
\Sigma_\mu(z) = \frac{\eta^{-1}_\mu(z)}{z}
\end{equation}
defined in the neighborhood is called the \emph{$\Sigma$-transform}. For two probability measures $\mu$ and $\nu$ on $\T$ with non-zero means it holds that
\[
\Sigma_{\mu\boxtimes\nu}(z)=\Sigma_\mu(z) \Sigma_\nu(z)
\]
on the common domain. If $\Sigma_\mu=\Sigma_\nu$ in a neighborhood of 0 then $\mu=\nu$. 

The definition of freeness yields that a $\boxtimes$-ID distribution is the normalized Haar measure if and only if its mean is zero. A $\boxtimes$-ID distribution $\mu$ with a non-zero mean has the \emph{free L\'evy-Khintchine representation} 
\begin{equation}\label{eq:freeMID}
\Sigma_{\mu}(z) = \exp \left( -i \alpha + \int_\T \frac{1+\xi z}{1-\xi z}\,\sigma(d\xi)  \right)  
\end{equation}
 in the domain of $\Sigma_\mu$, where $\alpha \in \R$ and $\sigma$ is a non-negative finite measure on $\T$. In this case the function $\Sigma_\mu$ can be extended to $\D$. The measure $\sigma$ is unique and $\alpha$ is unique up to translations by $2\pi \Z$. 
 
Conversely, given  $\alpha \in \R$ and a non-negative finite measure $\sigma$ on $\R$, there exists a $\boxtimes$-ID distribution $\mu$ such that \eqref{eq:freeMID} holds.  The pair $(\alpha,\sigma)$ is called a generating pair.

The main result of this section is the following. 
\begin{theorem}[Free L\'evy-Khintchine representation for $\boxtimes$-CHs] \label{thm:FLK}
For a continuous $\boxtimes$-CH $(\lambda_{s,t})_{t\ge s\ge0 }$ on $\T$ there exists a unique family $(\alpha_t,\sigma_t)_{t\ge0}$ with \ref{LK1} and \ref{LK2} and 
\begin{equation}\label{eq:FLKT}
\Sigma_{\lambda_{0,t}}(z) = \exp \left(-i \alpha_t + \int_\T \frac{1+\xi z}{1-\xi z}\,\sigma_t(d\xi)  \right), \qquad z \in \D,~t\ge0.
\end{equation}
Conversely, given a family $(\alpha_t,\sigma_t)_{t\ge0}$ satisfying \ref{LK1} and \ref{LK2} there exists a unique continuous $\boxtimes$-CH $(\lambda_{s,t})_{t\ge s\ge0 }$ on $\T$ for which \eqref{eq:FLKT} holds. The family $(\alpha_t,\sigma_t)_{t\ge0}$ is called the \emph{generating family} for the corresponding continuous $\boxtimes$-CH. 
\end{theorem}

As analogy to the classical case, Theorem \ref{thm:FLK} implies the free L\'evy-Khintchine representation for the increments 
\begin{equation}\label{eq:increments_f}
\Sigma_{\lambda_{s,t}}(z) = \exp \left(-i \alpha_{s,t} + \int_\T \frac{1+\xi z}{1-\xi z}\,\sigma_{s,t}(d\xi)  \right), \qquad z \in \D,~ 0 \le s \le t, 
\end{equation}
where $\alpha_{s,t}= \alpha_t -\alpha_s$ and $\sigma_{s,t}=\sigma_t - \sigma_s$. 

For the proof of Theorem \ref{thm:FLK} we prepare some lemmas.

\begin{lemma}\label{lem:free1} Let $(\lambda_{s,t})_{t\ge s\ge0 }$ be a continuous $\boxtimes$-CH. For every $t \ge s\ge 0$ the measure $\lambda_{s,t}$ is $\boxtimes$-ID and has a non-zero mean. 
\end{lemma}
\begin{proof}
We can prove that, for every $\epsilon,T>0$,  
\begin{equation}\label{eq:IA}
\lim_{\delta \to0^+} \sup_{\substack{0 \le s \le t \le T \\ 0\le t-s \le \delta}} \lambda_{s,t}(\{\xi \in\T: |\xi-1| \ge \epsilon\})=0,  
\end{equation}
because otherwise there would be a sequence $\delta_k\downarrow0$, $0\le s_k \le t_k \le T$ and an $\alpha>0$ such that $0 \le t_k - s_k \le \delta_k$ and 
\[
\lambda_{s_k,t_k}(\{\xi \in\T: |\xi-1| \ge \epsilon\}) \ge \alpha, \qquad k\in \N.  
\] 
By passing to a subsequence if necessary, we may assume that $s_k \to u$ and $t_k \to u$ for some $u\in[0,T]$. This would imply by Proposition \ref{prop:portmanteau_finite} that 
\[
\lambda_{u,u}(\{\xi \in\T: |\xi-1| \ge \epsilon\}) \ge \limsup_{k \to\infty}\lambda_{s_k,t_k}(\{\xi \in\T: |\xi-1| \ge \epsilon\}) \ge \alpha, 
\]
which is a contradiction to the fact $\lambda_{u,u}=\delta_1$.
 
The claim that $\lambda_{s,t}$ is $\boxtimes$-ID now follows from the free analogue of the Bawly and Khintchine theorem \cite[Theorem 1.1]{BB08} applied to the decomposition 
\[
\mu_{s,t} = \mu_{s, s+(t-s)/n}\boxtimes\mu_{s+(t-s)/n, s+2(t-s)/n}\boxtimes \mu_{s+2(t-s)/n, s+3(t-s)/n}\boxtimes \cdots \boxtimes \mu_{s+(n-1)(t-s)/n, t}
\]
 and that $\{\mu_{s+(k-1)(t-s)/n, s+k(t-s)/n}\}_{n\ge1, n \ge k \ge 1}$ is an infinitesimal triangular array thanks to \eqref{eq:IA}. 
 
 We are in a position to prove that each $\lambda_{s,t}$ has a non-zero mean $m(s,t)$. Suppose to the contrary that $m(s_0,t_0)=0$ for some $t_0 >s_0 \ge0$. Since $m(s_0,s_0)=1$, the minimum  
\[
t_1= \min\{t>s_0: m(s_0,t)=0\} \in (s_0,t_0] 
\]
exists. For $s_0 \le t < t_1$ we have $0=m(s_0,t_1)= m(s_0,t)m(t,t_1)$ and hence $m(t,t_1)=0$ since $m(s_0,t)\neq0$. By taking the limit $t\to t_1$ we obtain $1=m(t_1,t_1)=0$, a contradiction. 
\end{proof}
\begin{remark}
The same proof can be used to prove that for any continuous $\circledast$-CH $(\mu_{s,t})_{t\ge s\ge0 }$, each $\mu_{s,t}$ is $\circledast$-ID and has no idempotent factors (or  has non-zero mean), which is stated in \cite[Remark 5.6.18]{Heyer}. 
\end{remark}

To the authors' knowledge the following is not explicitly written in the literature, but can be easily guessed from a similar result for free additive convolution  \cite[Theorem 3.8]{BNT02}. 

\begin{lemma}\label{lem:free2} Let $\mu_,\mu_n, n=1,2,3,\dots$ be $\boxtimes$-ID distributions with non-zero mean and let $(\alpha,\sigma),(\alpha_n,\sigma_n)$ be their generating pairs respectively. Then $\mu_n$ converges to $\mu$ weakly as $n\to \infty$ if and only if $e^{i\alpha_n} \to e^{i\alpha}$ and $\sigma_n \to \sigma$ weakly as $n\to\infty.$
\end{lemma}
\begin{proof} Suppose that $e^{i\alpha_n} \to e^{i\alpha}$ and $\sigma_n \to \sigma$ weakly as $n\to\infty.$ Then $\Sigma_{\mu_n}(z)$ converges uniformly to $\Sigma_\mu$ locally uniformly on $\D$ by applying \cite[Lemma 2.11]{FHS18}.  The convergence $\mu_n\to\mu$ then follows from \cite[Proposition 2.9]{BV92}.   

Suppose that $\mu_n$ converges to $\mu$ weakly. Again by \cite[Proposition 2.9]{BV92} the function $\Sigma_{\mu_n}$ converges to $\Sigma_\mu$ uniformly in a neighborhood of 0. Since $\Sigma_{\mu_n}(0)= e^{-i\alpha_n}e^{\sigma_n(\T)}$, the sequence $e^{i\alpha_n}$ converges to $e^{i\alpha}$ and the sequence $\sigma_n(\T)$ is bounded. By the latter fact and Helly's theorem (or Prohorov's theorem), there exists a subsequence of $\sigma_n$,  denoted by $\sigma_{n(k)}$, that converges weakly to a finite non-negative measure $\sigma'$ on $\T$. For this subsequence $\Sigma_{\mu_{n(k)}}(z)$ converges to $\Sigma_{\mu'}(z)$ on $\D$, where $\mu'$ is the $\boxtimes$-ID distribution characterized by the generating pair $(\alpha,\sigma')$. Since $\Sigma_\mu= \Sigma_{\mu'}$ in a neighborhood of $0$, we must have $\mu= \mu'$ and hence $\sigma=\sigma'$ by the uniqueness of the generating pair. Those arguments show that $\sigma_n$ converges to $\sigma$ weakly. 
\end{proof}

\begin{proof}[Proof of Theorem \ref{thm:FLK}] By Lemma \ref{lem:free1} for each $t\ge s \ge0$ the measure $\lambda_{s,t}$ is $\boxtimes$-ID and is not the normalized Haar measure, and hence $\lambda_{s,t}$ has a generating pair $(\alpha_{s,t},\sigma_{s,t})$. The formula \eqref{eq:FLKT} holds for $\alpha_t=\alpha_{0,t}$ and $\sigma_t=\sigma_{0,t}$. Since $\lambda_{s,s}=\delta_1$ we obtain $e^{i\alpha_{s,s}}=1$ and $\sigma_{s,s}=0$. By Lemma \ref{lem:free2}, the mapping $(s,t)\mapsto \sigma_{s,t}$, and in particular the mapping $t\mapsto \sigma_t$, is weakly continuous, and we can take $\alpha_t$ such that $\alpha_0=0$ and the mapping $t\mapsto \alpha_t$ is continuous. The relation $\sigma_{s} + \sigma_{s,t}=\sigma_{t}$ holds because of the hemigroup property $\lambda_{0,s} \boxtimes \lambda_{s,t}=\lambda_{0,t}$ and the uniqueness of $\sigma_t$.  Therefore, for every Borel set $B\in \cB(\T)$ we have $\sigma_s(B) \leq \sigma_t(B)$, and since $\sigma_{s,s}(\partial B)=0$ we conclude by weak continuity that $\lim_{t\downarrow s}\sigma_{s,t}(B)=0$ and so $\sigma_t(B)\to \sigma_s(B)$ as $t\downarrow s$. A similar argument shows that $\sigma_t(B)\to \sigma_s(B)$ as $t\uparrow s$. In conclusion, the properties \ref{LK1} and \ref{LK2} are fulfilled. Uniqueness of the family $(\alpha_t,\sigma_t)_{t\ge0}$ easily follows from the uniqueness of $e^{i\alpha_t}$ and $\sigma_t$ for each $t\ge0$ and \ref{LK1}. 

Conversely, given a generating family $(\alpha_t,\sigma_t)_{t\ge0}$, for every $t \ge s \ge 0$ there exists a $\boxtimes$-ID distribution $\lambda_{s,t}$ whose generating pair is $(\alpha_t-\alpha_s, \sigma_t-\sigma_s)$. Those measures form a $\boxtimes$-CH. Continuity of the hemigroup follows by Lemma \ref{lem:free2}. 
\end{proof}

\begin{definition}\label{def:c-f} A bijection 
\[\Theta_F\colon\{\text{continuous $\circledast$-CHs on $\T$}\} \to \{\text{continuous $\boxtimes$-CHs on $\T$}\}
\]
is defined by sending the continuous $\circledast$-CH characterized by a generating family to the continuous $\boxtimes$-CH having the same generating family. 
\end{definition}

\begin{example} 
Then $\circledast$-convolution hemigroup $(N_\T(0,t-s))_{t\ge s \ge0}$ is characterized by the generating family $(0,(t/2)\delta_1)$ and is mapped by $\Theta_F$ to a $\boxtimes$-CH $(\nu_{t-s})_{t\ge s \ge0}$, where $\nu_t$ is the free normal distribution with mean $e^{-t}$ introduced and investigated in \cite{Bia97b,Bia97c}. 
\end{example}

Now we characterize the convergence of $\boxtimes$-CHs by means of convergence of generating families. This will add another object to Fig.\ \ref{fig:star} that is homeomorphic to the other objects. 

 \begin{proposition}\label{prop:conv_free}
Let $(\lambda_{s,t})_{t\ge s \ge 0}, (\lambda_{s,t}^{(n)})_{t\ge s \ge 0}$, $n=1,2,3,\dots$ be continuous $\boxtimes$-CHs, $(\eta_{s,t})_{t\ge s \ge 0}$, $(\eta_{s,t}^{(n)})_{t\ge s \ge 0}$ be their $\eta$-transforms, and $(\alpha_t,\sigma_t)_{t\ge0}$,  $(\alpha_t^{(n)},\sigma_t^{(n)})_{t\ge0}$ be their generating families, respectively. The following conditions are equivalent as $n\to\infty$: 
\begin{enumerate}[label=\rm(F\arabic*)]
\item\label{condF4} $(\eta_{0,t}^{(n)})_{t\ge0}$ converges to $(\eta_{0,t})_{t\ge0}$ locally uniformly on $\D \times [0,\infty)$; 

\item\label{condF5} $(\eta_{s,t}^{(n)})_{t\ge s \ge 0}$ converges to $(\eta_{s,t})_{t\ge s \ge 0}$ locally uniformly on $\D \times \{(s,t): 0 \le s \le t <\infty\}$; 

\item\label{condF1} $(\lambda_{0,t}^{(n)})_{t\ge0}$ weakly converges to $(\lambda_{0,t})_{t\ge0}$ locally uniformly on $[0,\infty)$; 

\item\label{condF2} $(\lambda_{s,t}^{(n)})_{t\ge s \ge 0}$ weakly converges to $(\lambda_{s,t})_{t\ge s \ge 0}$ locally uniformly on the index set $\{(s,t): 0 \le s \le t <\infty\}$;

\item \label{condF3} $(\alpha_t^{(n)})_{t\ge0}$ converges to $(\alpha_t)_{t\ge0}$ locally uniformly on $[0,\infty)$ and the mutually equivalent conditions \ref{condG1}--\ref{condG4} in Proposition \ref{prop:convergence_equivalence} hold. 
 \end{enumerate} 
\end{proposition}
\begin{proof} We will establish the equivalence of \ref{condF1}, \ref{condF2} and \ref{condF3}. Once we establish it, then the proof of \ref{condF2} $\Rightarrow$ \ref{condF5} $\Rightarrow$ \ref{condF4} $\Rightarrow$ \ref{condF1} will be almost the same as the arguments \ref{condM3} $\Rightarrow$ \ref{condM1} and \ref{condM2} $\Rightarrow$ \ref{condM4} in Theorem \ref{thm:conv_monotone}. 

\vspace{2mm}\noindent\underline{\ref{condF2} $\Rightarrow$ \ref{condF1}}:  obvious. 

\vspace{2mm}\noindent\underline{\ref{condF1} $\Rightarrow$ \ref{condF3}.}  Because $\lambda_{0,t}^{(n)}$ weakly converges to $\lambda_{0,t}$ for each $t\ge0$, condition \ref{condG2} follows by Lemma \ref{lem:free2}. Since 
\[
e^{i\alpha_t^{(n)}} = \frac{\int_\T \xi \,d\lambda_{0,t}^{(n)}(\xi)}{\left| \int_\T \xi \,d\lambda_{0,t}^{(n)}(\xi)\right|} \qquad \text{and} \qquad e^{i\alpha_t} =  \frac{\int_\T \xi \,d\lambda_{0,t}(\xi)}{\left| \int_\T \xi \,d\lambda_{0,t}(\xi)\right|}, 
\]
 the locally uniform convergence of $\exp[i\alpha_t^{(n)}]$ to $\exp[i\alpha_t]$ holds, and hence of $\alpha_t^{(n)}$ to $\alpha_t$ holds. 

\vspace{2mm}\noindent\underline{\ref{condF3} $\Rightarrow$ \ref{condF2}.} Denote by $m_k(\sigma)$ the $k$-th moment $\int_{\T} \xi^k \,d\sigma$ of a finite measure $\sigma$. The assumption implies that $\alpha_{s,t}^{(n)}$ converges to $\alpha_{s,t}$ locally uniformly and $m_k(\sigma_{s,t}^{(n)})$ converges to $m_k(\sigma_{s,t})$ locally uniformly for every $k \in \N$, where $\alpha_{s,t}=\alpha_t-\alpha_s$, $\sigma_{s,t}=\sigma_t-\sigma_s$ and $\alpha_{s,t}^{(n)},\sigma_{s,t}^{(n)}$ are similarly defined. By examining the relations between $\Sigma_\lambda, \eta_\lambda$ and $\psi_\lambda$, one sees that for each $k \in \N$ there is a universal polynomial $P_k(x_1,x_2,\dots, x_k)$ independent of $\lambda$ such that 
\[
m_k(\lambda)= P_k (e^{i \alpha- \sigma(\T)}, m_1(\sigma), \dots, m_{k-1}(\sigma))  
\]
for every $\boxtimes$-ID distribution $\lambda$ with a generating pair $(\alpha,\sigma)$; for example 
\[
m_1(\lambda) = e^{i \alpha- \sigma(\T)}, \qquad m_2(\lambda) =  (e^{i \alpha- \sigma(\T)})^2(1-2 m_1(\sigma)). 
\]
Hence $m_k(\lambda_{s,t}^{(n)})$ converges to $m_k(\lambda_{s,t})$ locally uniformly on $(s,t)$ for each $k \ge0$ and hence for $k \in \Z$ by complex conjugate. The use of Stone-Weierstrass' theorem gives the desired conclusion. 
\end{proof}

\begin{corollary} The bijection $\Theta_F$ is a homeomorphism with respect to the locally uniform weak convergence of convolution hemigroups. 
\end{corollary}

\subsection{Interpretation of the bijection: generator of moments} \label{sec:interpretation}

This section exhibits an interpretation of the bijection in Definition \ref{def:c-f} in terms moments. Suppose that $(\mu_{s,t})_{t\ge s \ge 0}$ is a continuous $\circledast$-CH as in Section \ref{sec:monotone} such that its generating family $(\alpha_t,\sigma_t)_{t\ge0}$ is given as  \eqref{eq:gen}. 
Let $(\lambda_{s,t})_{t\ge s \ge 0}= \Theta_F((\mu_{s,t})_{t\ge s \ge 0})$; namely, it is a continuous $\boxtimes$-CH characterized by the same generating family. Let $\eta_{s,t}$ be the $\eta$-transform of $\lambda_{s,t}$. Since $\eta_{s,t}^{-1}(z) = z\Sigma_{\lambda_{s,t}}(z)$ we have by \eqref{eq:increments_f} 
\[
\partial_t \eta_{s,t}^{-1}(z)|_{t=s} = z \left[-i \gamma_s + \int_\T \frac{1+\xi z}{1- \xi z} d\rho_s(\xi) \right]. 
\]
Taking the derivative of the identity $\eta_{s,t}^{-1} (\eta_{s,t}(z))=z$ gives $\partial_t \eta_{s,t}(z)|_{t=s}= - \partial_t \eta_{s,t}^{-1}(z)|_{t=s} $. Switching to the moment generating function $\psi_{s,t}(z)=\eta_{s,t}(z)/(1-\eta_{s,t}(z))$ we obtain 
\[
\partial_t \psi_{\lambda_{s,t}}(z)|_{t=s} = \frac{z}{(1-z)^2} \left[i \gamma_s - \int_\T \frac{1+\xi z}{1- \xi z} d\rho_s(\xi) \right]. 
\]
Following the corresponding computations in Section \ref{sec:monotone}  we arrive at 
\begin{equation*}
 \left. \frac{\partial}{\partial t}\right|_{t=s} \int_{\T} \xi^n \,d\lambda_{s,t}(\xi)=  \left. \frac{\partial}{\partial t}\right|_{t=s} \int_\T \xi^n \,d\mu_{s,t}(\xi), \qquad n \in \Z. 
\end{equation*}
Therefore our bijection between convolution hemigroups is defined so that the ``generators'' of moments coincide.

\subsection{Boolean convolution hemigroups on the unit circle} 
The results in Section \ref{sec:c-f2}--\ref{sec:interpretation} have analogy for boolean convolution. 
The \emph{multiplicative boolean convolution} $\mu \utimes \nu$ of probability measures on the unit circle was defined in \cite{Fra09} and was characterized by 
$$
z\eta_{\mu \sutimes \nu}(z) = \eta_\mu(z) \eta_\nu(z), \qquad z \in \D.   
$$
 If $\mu$ is $\utimes$-ID with $\int_\T \xi \,d\mu(\xi)\neq0$ then the \emph{boolean L\'evy-Khintchine representation}  
\begin{equation}\label{eq:BID}
\eta_{\mu}(z) = z \exp \left( i \alpha - \int_\T \frac{1+\xi z}{1-\xi z}\,\sigma(d\xi)  \right)
\end{equation}
holds on $\D$, where $\alpha \in \R$ and $\sigma$ is a non-negative finite measure on $\T$. The measure $\sigma$ is unique and $\alpha$ is unique up to translations by $2\pi \Z$ as in the free case. Conversely, given  $\alpha \in \R$ and a non-negative finite measure $\sigma$ on $\R$, there exists a $\utimes$-ID distribution $\mu$ such that \eqref{eq:BID} holds.  The pair $(\alpha,\sigma)$ is called a generating pair of $\mu$. 

One can prove the following boolean analogue of Theorem \ref{thm:FLK}. 
\begin{theorem}[Boolean L\'evy-Khintchine representation for $\utimes$-CHs] \label{thm:BLK}
For a continuous $\utimes$-CH $(\zeta_{s,t})_{t\ge s\ge0 }$ on $\T$ there exists a unique family $(\alpha_t,\sigma_t)_{t\ge0}$ with \ref{LK1} and \ref{LK2} and 
\begin{equation}\label{eq:BLKT}
\eta_{\zeta_{0,t}}(z) = z\exp \left(i \alpha_t - \int_\T \frac{1+\xi z}{1-\xi z}\,\sigma_t(d\xi)  \right), \qquad z \in \D,~t\ge0.
\end{equation}
Conversely, given a family $(\alpha_t,\sigma_t)_{t\ge0}$ satisfying \ref{LK1} and \ref{LK2} there exists a unique continuous $\utimes$-CH $(\zeta_{s,t})_{t\ge s\ge0 }$ on $\T$ for which \eqref{eq:BLKT} holds. The family $(\alpha_t,\sigma_t)_{t\ge0}$ is called the \emph{generating family} for the corresponding continuous $\utimes$-CH. 
\end{theorem}

The proof is quite similar to the free case, with the help of the following facts. 

\begin{lemma} Let $(\zeta_{s,t})_{t\ge s\ge0 }$ be a continuous $\utimes$-CH. For every $t \ge s\ge 0$ the measure $\zeta_{s,t}$ is $\utimes$-ID and has a non-zero mean. 
\end{lemma}

\begin{proof}
The proof of Lemma \ref{lem:free1} works; we only need to use \cite[Theorem 3.4]{Wan08} instead of \cite[Theorem 1.1]{BB08}. 
\end{proof}

\begin{lemma}\label{lem:boole2} Let $\mu_,\mu_n, n=1,2,3,\dots$ be $\utimes$-ID distributions with non-zero mean and let $(\alpha,\sigma),(\alpha_n,\sigma_n)$ be their generating pairs respectively. Then $\mu_n$ converges to $\mu$ weakly as $n\to \infty$ if and only if $e^{i\alpha_n} \to e^{i\alpha}$ and $\sigma_n \to \sigma$ weakly as $n\to\infty.$
\end{lemma}
\begin{proof}
The proof is similar to Lemma \ref{lem:free2}. We can use  \cite[Lemma 2.11]{FHS18} instead of \cite[Proposition 2.9]{BV92}. 
\end{proof}

Imitating Definition \ref{def:c-f} provides with a bijection between classical and boolean continuous CHs. It is a homeomorphism by the boolean analogue of Proposition \ref{prop:conv_free}, so that the new object of continuous $\utimes$-CHs can be added to Fig.\ \ref{fig:star}. The proof is similar and details are omitted. Moreover, under some assumptions on $\alpha_t$ and $\sigma_t$ we can see that this bijection identifies the time-derivative of moments of two convolution hemigroups in a way similar to Section \ref{sec:interpretation}.

\subsection{Embedding of free convolution hemigroups into monotone ones}\label{sec:subordination}

As pointed out by Franz \cite[Corollary 5.3]{Fra09a} after the pioneering work by Biane \cite{Bia98}, a subordination property for free convolution yields an embedding of the set of free CHs into the set of monotone ones.   Related work can be found in \cite{Sch17}, \cite[Sections 4.7 and 5.5]{FHS18}, \cite{Jek20}. 

Let $(\lambda_{s,t})_{t \ge s \ge 0}$ be a continuous $\boxtimes$-CH. There exists a unique probability measure $\mathring{\lambda}_{s,t}$ such that 
$$
\lambda_{0,t}= \lambda_{0,s} \circlearrowright \mathring{\lambda}_{s,t}, \qquad t \ge s \ge 0.    
$$
The family $(\mathring{\lambda}_{s,t})_{t \ge s \ge 0}$ forms a $\circlearrowright$-CH because each mapping $\eta_{\lambda_{0,t}}$ is univalent as a consequence of Lemma \ref{lem:free1} and \cite[Proposition 7.12]{FHS18}. The weak continuity of $(s,t)\mapsto \mathring{\lambda}_{s,t}$ is a consequence of the representation $\eta_{\mathring{\lambda}_{s,t}}= \eta_{\lambda_{0,s}}^{-1} \circ \eta_{\lambda_{0,t}}$ and \cite[Lemma 2.11]{FHS18} and Proposition \ref{prop:inverse_convergence}. 

In this section we compute the generating family of  $(\mathring{\lambda}_{s,t})_{t \ge s \ge 0}$.

\begin{proposition} 
Let $(\lambda_{s,t})_{t \ge s \ge 0}$ be a continuous $\boxtimes$-CH with generating family $(\alpha_t,\sigma_t)_{t\ge0}$ and let $\Sigma$ be the measure on $\T \times [0,\infty)$ associated to $(\sigma_t)_{t\ge0}$. Then the generating family $(\mathring\alpha_t,\mathring\sigma_t)_{t\ge0}$ for the continuous $\circlearrowright$-CH $(\mathring{\lambda}_{s,t})_{t \ge s \ge 0}$ is given by 
\[
\mathring{\alpha}_t = \alpha_t \quad \text{and} \quad \mathring\sigma_t(A) = \int_{\T \times [0,t]} (\delta_\xi \circlearrowright \lambda_{0,s})(A) \Sigma(d\xi ds), \qquad t\ge0,~ A \in \cB(\T). 
\] 
\end{proposition}

\begin{proof} 
Let $\eta_t:= \eta_{\lambda_{0,t}} = \eta_{\mathring \lambda_{0,t}}$. One can see from the definition of the $\Sigma$-transform that $\eta_t'(0)= e^{i\alpha_t - \sigma_t(\T)}$. By the definition of the generating family for monotone CHs, we conclude that $\mathring \alpha_t = \alpha_t$. 

Assume that there are continuous functions $\gamma_t$ and $f(\xi,t)$ such that  
\[
\alpha_t = \int_0^t \gamma_s\,ds \qquad  \text{and}\qquad \sigma_t(A) = \int_{A \times [0,t]} f(\xi,s)\,d\xi ds , \qquad A \in \cB(\T). 
\]
Note that in this case we obtain $\Sigma(d\xi dt)=f(\xi,t)d\xi dt$.  Write $\eta_t^{-1}(z) = z \exp(u_t(z))$, where 
\[
u_t(z) =-i \alpha_t + \int_\T \frac{1+\xi z}{1-\xi z}\,\sigma_t(d\xi). 
\]
After some calculus around the relation $\eta_t^{-1}(z) = z e^{u_t(z)}$, we arrive at 
\[
\mathring p(z,t):= -\frac{1}{z}\cdot \frac{\partial_t \eta_t(z)}{\partial_z \eta_t(z)} = (\partial_t u_t) (\eta_t(z)),  
\] 
so that $-z\mathring p(z,t)$ is the Herglotz vector field for the decreasing Loewner chain $(\eta_t)_{t\ge0}$. Now we proceed as 
\[
\mathring p(z,t) = -i \gamma_t + \int_\T \frac{1+\xi \eta_t(z)}{1-\xi \eta_t(z)} f(\xi,t)\,d\xi. 
\]
Using the identity
\begin{equation*}\label{eq:identity}
\frac{1+\xi \eta_t(z)}{1-\xi \eta_t(z)} = \int_{\T} \frac{1+ w z}{1- w z} (\delta_\xi \circlearrowright \lambda_{0,t})(dw)
\end{equation*}
 we arrive at 
\[
\int_0^t \mathring p(z,s) \,ds = -i \alpha_t + \int_{\T\times [0,t]}\left[ \int_\T \frac{1+w z}{1-w z}(\delta_\xi \circlearrowright \lambda_{0,s})(dw)\right] \Sigma(d\xi ds), \qquad z \in \D. 
\]
In view of \eqref{eq:abs_cont2}, we obtain 
\begin{equation}\label{eq:identity1}
\int_\T   \frac{1+w z}{1-w z}   d\mathring \sigma_t(w) = \int_{\T\times [0,t]}\left[ \int_\T \frac{1+w z}{1-w z}(\delta_\xi \circlearrowright \lambda_{0,s})(dw)\right] \Sigma(d\xi ds). 
\end{equation}
By the inversion formula (see e.g.\ \cite[equation (8) in Theorem IV.14, p.145]{Tsuji:1975} or \cite[Lemma 2.8]{FHS18}), the measure $\mathring \sigma_t$ is given as desired. 

In the general case, we approximate $\alpha_t$ and $\Sigma$ by smooth functions $\alpha_t^{(n)}$ and measures $\Sigma^{(n)}$ with smooth densities respectively such that $\alpha_t^{(n)}\to \alpha_t$ locally uniformly and $\sigma_t^{(n)} \to\sigma_t$ weakly for each $t\ge0$. Let $(\lambda_{s,t}^{(n)})_{t\ge s\ge 0}$ denote the $\boxtimes$-CH corresponding to the generating family $(\alpha_t^{(n)}, \sigma_t^{(n)})_{t\ge0}$. Since $\mathring\lambda_{0,t}^{(n)} = \lambda_{0,t}^{(n)}$ and $\mathring\lambda_{0,t} = \lambda_{0,t}$, the family $(\mathring\lambda_{0,t}^{(n)})_{t\ge0}$ converges weakly locally uniformly to $(\mathring\lambda_{0,t})_{t\ge0}$ by Proposition \ref{prop:conv_free}. Therefore, we infer from Theorem \ref{thm:conv_monotone} that $\mathring\sigma_t^{(n)} \to \mathring\sigma_t$ weakly for each $t$. 

We have already established in \eqref{eq:identity1} that 
\begin{equation}\label{eq:approx}
\int_\T  \frac{1+w z}{1-w z}  \,d\mathring\sigma_t^{(n)}(w) = \int_{\T \times [0,t]} \left[ \int_\T  \frac{1+w z}{1-w z}  \,d (\delta_\xi \circlearrowright \lambda_{0,s}^{(n)})(w)\right] \Sigma^{(n)}(d\xi ds). 
\end{equation}
 By Proposition \ref{prop:conv_free} we know that 
\[
\int_\T  \frac{1+w z}{1-w z}  \,d (\delta_\xi \circlearrowright \lambda_{0,s}^{(n)})(w) = \frac{1+\xi \eta_{\lambda_{0,s}^{(n)}}(z)}{1-\xi \eta_{\lambda_{0,s}^{(n)}}(z)} \to \frac{1+\xi \eta_{\lambda_{0,s}}(z)}{1-\xi \eta_{\lambda_{0,s}}(z)} = \int_\T  \frac{1+w z}{1-w z}  \,d (\delta_\xi \circlearrowright \lambda_{0,s})(w)
\]
 uniformly for $(\xi,s) \in \T \times [0,t]$ and a fixed $z$ as $n\to\infty$. Now, letting $n$ tend to infinity in \eqref{eq:approx} yields 
\begin{equation*}
\int_\T  \frac{1+w z}{1-w z}  \,d\mathring\sigma_t(w) = \int_{\T \times [0,t]} \left[ \int_\T  \frac{1+w z}{1-w z}  \,d (\delta_\xi \circlearrowright \lambda_{0,s})(w)\right] \Sigma(d\xi ds),   
\end{equation*}
which leads to the desired conclusion again by the inversion formula. 
\end{proof}

\begin{remark}
Recently Biane proved that $(\mathring{\lambda}_{s,t})_{t\ge s \ge 0}$ is time-homogeneous if and only if $(\lambda_{0,t})_{t\ge0}$ is a convolution semigroup of Poisson kernels or delta measures \cite[Theorem 5.3]{Bia19}. An equivalent statement is that the generating family $(\mathring\alpha_t,\mathring\sigma_t)_{t\ge0}$ above is of the form $\mathring\alpha_t= t \mathring \alpha$ and $\mathring\sigma_t = t \mathring \sigma$ if and only if $\alpha_t=t\alpha$ and $\sigma_t = t c h$ for some $\alpha \in \R$ and $c \ge0$, where $h$ is the normalized Haar measure on $\T$.  In this case we also have $\alpha =\mathring \alpha$ and $\mathring \sigma = ch$. 
\end{remark}

\appendix 

\section{Measure theory}\label{sec:measure}

This section collects basic notions and supplementary facts in measure theory. For further basic notions, see Section \ref{sec:measure0}.  

In this appendix, let $(S,d)$ be a polish space (= complete separable metric space), while some results hold under weaker assumptions (for example, Propositions \ref{prop:portmanteau_finite} and \ref{prop:portmanteau} hold on any metric space $S$).  The function 
\[
d_P(\sigma, \tau) = \inf\{\epsilon>0: \sigma(B) \le \tau(B^\epsilon)+\epsilon \text{~and~} \tau(B) \le \sigma(B^\epsilon)+\epsilon \text{~for all $B\in\cB(S)$}\},  
\]
where $B^\epsilon= \{x \in S: d(x,B) <\epsilon\}$ as before, with the convention that $d(x,\emptyset)=\infty$, defines a metric, called the \emph{Prohorov metric}, on the set of finite non-negative measures on $S$. The Prohorov metric is compatible with the weak convergence, that is, for finite measures $\sigma,\sigma_n (n\ge1)$ on $S$, $\sigma_n$ converges weakly to $\sigma$ if and only if $d_P(\sigma_n,\sigma)$ converges to 0. For the proof, the reader is referred to \cite[Lemma 4.3]{Kal} or \cite[Theorems 11.3.1, 11.3.3]{Dud02}, the latter of which treats the case of probability measures but the proofs can be easily adapted to finite non-negative measures. 

Prohorov's theorem characterizes the relative compactness (or equivalently, sequential compactness thanks to the Prohorov metric) for probability measures (see e.g.\ \cite[Theorem 11.5.4]{Dud02} or \cite[Theorems 6.1 and 6.2]{Bil68}). It can be easily extended to finite non-negative measures. A family $\cF$ of finite non-negative measures on $S$ is referred to as \emph{tight} if for every $\epsilon>0$ there exists a compact subset $K$ such that $\sup_{\sigma \in \cF} \sigma(S\setminus K) < \epsilon$. We say that $\cF$ is \emph{uniformly bounded} if $\sup_{\sigma \in \cF}\sigma(S)<\infty$. 
By definition, if $S$ is compact then any family of finite non-negative measures is tight.

\begin{theorem}[Prohorov's theorem]\label{thm:Prohorov} A family of finite non-negative measures on $S$ is relatively compact if and only if it is tight and uniformly bounded. 
\end{theorem}
\begin{proof}
See \cite[Theorem 4.2]{Kal}, or alternatively, using Prohorov's theorem for probability measures, one can also give a short proof by normalizing measures by their total masses. 
\end{proof}

The Portmanteau theorem provides a characterization of convergence of probability measures \cite[Theorem 2.1]{Bil68}. It can be extended to finite non-negative measures. 
 
\begin{proposition}[Portmanteau theorem]\label{prop:portmanteau_finite} Let $\sigma, \sigma_n, n=1,2,3,\dots$ be  finite non-negative measures on $S$. The following statements are equivalent. 

\begin{enumerate}[label=\rm(\arabic*)]
\item\label{item:conv_finite1} $\sigma_n$ converges weakly to $\sigma$, 

\item\label{item:conv_finite2} for every closed subset $C$ of $S$ and open subset $G$ of $S$
\[
\limsup_{n\to\infty} \sigma_n(C) \le \sigma(C) \qquad \text{and} \qquad \liminf_{n\to\infty}\sigma_n(G) \ge \sigma(G), 
\]

\item \label{item:conv_finite3} 
for every Borel subset $B$ of $S$ such that $\sigma(\partial B)=0$,  
\[
\sigma(B) = \lim_{n\to\infty} \sigma_n(B).  
\]
\end{enumerate}
\end{proposition}
\begin{remark} 
For probability measures, one inequality of the two of \ref{item:conv_finite2} implies the other. For finite measures, only one inequality does not guarantee the convergence of the total mass, so both inequalities are needed. 
\end{remark}
\begin{proof}
Note that any of the three statements implies that the total mass $\sigma_n(S)$ converges to $\sigma(S)$. Two cases are possible: $\sigma(S)=0$ and $\sigma(S)>0$. In the former case the proof is easier, and in the latter case one can reduce the problem to probability measures by rescaling.  Alternatively, one can also modify the proof of Proposition \ref{prop:portmanteau} below. The details are left to the reader. 
\end{proof}

A Portmanteau-type theorem also holds for vague convergence as follows. Some part is known in  \cite[Lemma 4.1]{Kal} but we give a complete proof for the sake of convenience.

\begin{proposition}\label{prop:portmanteau} Let $\rho, \rho_n, n=1,2,3,\dots$ be locally finite non-negative measures on $S$. The following statements are equivalent. 

\begin{enumerate}[label=\rm(\arabic*)]
\item\label{item:conv1} $\rho_n$ converges vaguely to $\rho$, 

\item\label{item:conv2} for every open bounded subset $G$ of $S$ and closed bounded subset $C$ of $S$
\[
\rho(G)  \le  \liminf_{n\to\infty}\rho_n(G) \qquad \text{and} \qquad  \rho(C) \ge \limsup_{n\to\infty} \rho_n(C), 
\]

\item \label{item:conv3} 
for every $B \in \cB_b(S)$ such that $\rho(\partial B)=0$  
\[
\rho(B) = \lim_{n\to\infty} \rho_n(B).  
\]
\end{enumerate}
\end{proposition} 
\begin{proof} 
\underline{\ref{item:conv1} $\Rightarrow$ \ref{item:conv2}.}  
For a closed bounded subset $C$ of $S$, take continuous functions $f_k, k=1,2,3,\dots$ such that $  \mathbf1_{C} \le f_k \le 1$, $f_k \to \mathbf1_{C}$ pointwisely, and the supports of $f_k$ are uniformly bounded. For example,  
set $f_k(x) = 0 \vee [1-k d(x,C)]$. 
Then we have
$$
\limsup_{n\to\infty}\rho_n(C) \le \limsup_{n\to\infty}\int_S f_k(x)\,d\rho_n(x)  = \int_S f_k(x)\,d\rho(x), 
$$ 
and the dominated convergence theorem yields that the RHS converges to $\rho(C)$ as $k\to\infty$. 

For an open bounded subset $G$ of $S$,  take continuous functions $g_k, k=1,2,3,\dots$ such that $0 \le g_k \le \mathbf 1_G$, $g_k \to \mathbf1_G$ pointwisely, and their supports are uniformly bounded. For example, set $g_k(x)= 1 \wedge d(x,S\setminus G)^{\frac1{k}}$. Then 
$$
\liminf_{n\to\infty}\rho_n(G) \ge \liminf_{n\to\infty}\int_S g_k(x)\,d\rho_n(x)  = \int_S g_k(x)\,d\rho(x), 
$$ 
and the dominated convergence theorem yields that the RHS converges to $\rho(G)$ as $k\to\infty$.

\vspace{2mm}
\noindent
\underline{\ref{item:conv2} $\Rightarrow$ \ref{item:conv3}.} For every $B\in\cB_b(S)$ such that $\rho(\partial B)=0$ we have
$
 \rho(\mathring{B})= \rho(B) =\rho(\overline{B}), 
$
where $\mathring{B}$ is the interior of $B$ and $\overline{B}$ is the closure of $B$. Since $\mathring{B}$ is open bounded and $\overline{B}$ is closed bounded,  
\begin{align*}
\rho(B) &= \rho(\mathring{B}) \le \liminf_{n\to\infty} \rho_n(\mathring{B}) \le \liminf_{n\to\infty} \rho_n(B) \\
& \le\limsup_{n\to\infty} \rho_n(B)\le \limsup_{n\to\infty} \rho_n(\overline{B}) \le \rho(\overline{B}) = \rho(B), 
\end{align*}
so that the desired conclusion follows.

\vspace{2mm}
\noindent
\underline{\ref{item:conv3} $\Rightarrow$ \ref{item:conv1}.} The arguments follow the lines of the proof of \cite[Lemma 4.1]{Kal}. Let $f\colon S\to \C$ be a bounded continuous function with bounded support. We may assume that $f\ge0$. Set the notation $\{f>t\}:=\{x \in S: f(x) >t\}$ and $\{f=t\}:=\{x \in S: f(x) = t\}$ for $t\ge0$ and $\|f\|_\infty:=\sup \{|f(x)|: x\in S\} .$ 
We will use the well known formula 
\[
\int_S f(x)\,d\rho(x) =\int_{[0,\infty)} \rho(\{f>t\})\,dt
\]
which follows from Tonelli's theorem applied to the function $F(x,t):=\mathbf1_{(0,\infty)}(f(x)-t). $

For all $t\ge0$, it holds that $\partial \{f>t\} \subset \{f=t\}$. Since 
$
\bigcup_{t >0}\{f=t\} 
$
is a support of $f$ and is bounded, it is of finite mass with respect to $\rho$, so that there are at most countably many $t > 0$ for which $\rho(\partial \{f>t\})>0$. In particular, by the assumption \ref{item:conv3}, $\lim_{n\to\infty}\rho_n(\{f>t\}) = \rho(\{f>t\})$ for a.e.\ $t\ge0$. By a similar argument, we can find a ball $B=\{x \in S: d(x_0,x)<r\}$ with $r>0$ such that $\{f>0\} \subset B$ and $\rho(\partial B)=0$. 
Observe that $\rho_n(\{f>t\}) \le \rho_n(\{f>0\}) \le \rho_n (B)$ for all $n \in \N$ and $t\ge0$, and $\sup_{n\in\N}\rho_n(B)<\infty$ since $\lim_{n\to\infty}\rho_n(B)=\rho(B)$. By the dominated convergence theorem, 
\[
\int_S f(x)\,d\rho_n(x) = \int_{[0, \|f\|_\infty]} \rho_n(\{f>t\})\,dt \to \int_{[0,\|f\|_\infty]} \rho(\{f>t\})\,dt =\int_S f(x)\,d\rho(x),  
\]
as desired. 
 \end{proof}

We close this section by proving a simple fact on the sum of measures.

\begin{lemma}\label{lem:limits}
Let $(X, \mathcal F)$ be a measurable space. For a sequence of measures $\{\mu_i\}_{i\ge1}$ on $(X, \mathcal F)$ define a measure $\mu :=\sum_{i\ge1} \mu_i$. 

\begin{enumerate}[label=\rm(\arabic*)]
\item\label{item:limits1} For every measurable function $f\colon X \to [0,\infty] $ we have 
$
\int_X f d\mu= \sum_{i\ge1} \int_X f d\mu_i. 
$
\item\label{item:limits2} Let $g \colon X \to \C$ be $\mu$-integrable (or equivalently, $ \sum_{i\ge1} \int_X |g| d\mu_i<\infty$, due to \ref{item:limits1}). Then we have $
\int_X g d\mu= \sum_{i\ge1} \int_X g d\mu_i. 
$
\end{enumerate}
\end{lemma}
\begin{proof}
The second statement is an obvious consequence of the first one. For the first statement, take a sequence of nonnegative simple functions $f_n= \sum_{k=1}^{d_n} a_{n,k} 1_{A_k}$ such that $f_n\uparrow f$. 
For every $m \in \N$ we obtain 
\begin{align*}
\int_X f d\mu &= \lim_{n\to\infty} \sum_{k=1}^{d_n} a_{n,k} \mu(A_k) =  \lim_{n\to\infty} \sum_{i=1}^\infty \sum_{k=1}^{d_n} a_{n,k} \mu_i(A_k) \\
&\ge  \lim_{n\to\infty} \sum_{i=1}^m \sum_{k=1}^{d_n} a_{n,k} \mu_i(A_k) =  \sum_{i=1}^m \int_X f d\mu_i,  
\end{align*}
and hence 
$$
\int_X f d\mu \ge \sum_{i\ge1} \int_X f d\mu_i. 
$$
On the other hand, 
$$
\int_X f d\mu  =  \lim_{n\to\infty} \sum_{i=1}^\infty \sum_{k=1}^{d_n} a_{n,k} \mu_i(A_k) =  \lim_{n\to\infty} \sum_{i=1}^\infty \int_X f_n d\mu_i \leq\sum_{i=1}^\infty \int_X f d\mu_i,  
$$
which completes the proof. 
\end{proof}

\section{Univalent mappings}\label{sec:univalent}

\begin{proposition}\label{prop:inverse_convergence} Let $f, f_n\colon \D \to \C$ be univalent mappings for $n=1,2,3,\dots$ such that $f_n \to f$ locally uniformly. 
 \begin{enumerate}[label=\rm(\arabic*)]
\item For every compact subset $K$ of $f(\D)$ there exists $n_0 \in \N$ such that $K \subset f_n(\D)$ for all $n \ge n_0$. 
\item $f_n^{-1}\colon f_n(\D)\to \D$ converges locally uniformly on $f(\D)$. 
\end{enumerate}
\end{proposition}
\begin{proof} The first statement follows from the kernel convergence of the ranges $f_n(\D)$; see \cite[Problem 3, p.31]{Pom75} or \cite[Theorem 6.1]{Yan}, the latter of which contains a detailed proof. Note that we can choose 0 as a reference point and we may assume that $f_n(0)=f(0)=0$ and $f'(0), f_n'(0)>0$ by considering $[f_n(z)-f_n(0)]/f_n'(0)$ and $[f(z)-f(0)]/f'(0)$. 

For the second statement, take the compact subset $K$ of $f(\D)$ above. There exists $r\in (0,1)$ such that $K \subset f(r\D) $ and $K \subset f_n(r\D)$ for all $n\ge n_0$. By the Lagrange inversion formula, we have
\[
f_n^{-1}(w) = \frac{1}{2\pi i} \int_{r\T} \frac{z f_n'(z)}{f_n(z)-w}\,dz, \qquad w \in K,   n \ge n_0 
\]
and a similar formula for $f^{-1}$. The desired conclusion readily follows from those formulas. 
\end{proof}

\begin{theorem}\label{thm:estimate_LC}
For a multiplicative Loewner chain $(f_{t})_{t \ge 0}$, we have
\begin{equation}
\label{eq:estimate_LC1}
|f_{s}(z)-f_{t}(z)| \le \frac{8\alpha|z|}{(1-|z|)^{4}} + \frac{4 |\beta| |z|}{(1-|z|)^{3}},\qquad 0 \le s \le t, ~z \in \D, 
\end{equation}
where
\[
\alpha := \Re \frac{f_{s}'(0)-f_{t}'(0)}{f_{s}'(0)+f_{t}'(0)}\ge0, \quad \beta := \Im \frac{f_{s}'(0)-f_{t}'(0)}{f_{s}'(0)+f_{t}'(0)} \in \R.
\]
In particular, if $f_{t}'(0) > 0$ for all $t \ge 0$, then
\begin{equation}
\label{eq:estimate_LC2}
|f_{s}(z)-f_{t}(z)| \le \frac{8|z|}{(1-|z|)^{4}}\left( \frac{1}{f_t'(0)} - \frac{1}{f_s'(0)}\right), \qquad 0 \le s \le t, ~z \in \D.
\end{equation}
\end{theorem}

\begin{proof}
Let $f_{st} := f_{s}^{-1}\circ f_{t}$. 
First we estimate $|z- f_{st}(z)|$. By the Schwarz lemma, the inequality $|f_{st}(z)/z|\le1$ holds and hence 
$$
q(z):= \frac{z- f_{st}(z)}{z + f_{st}(z)}: \D \to \{\Re \,z \ge 0\}, 
$$
where $q(0):=\alpha+i\beta $. Thus the Herglotz representation \cite[Theorem 2.4]{Pom75}
\[
q(z) = i\beta + \int_\T \frac{\xi+z}{\xi -z} \,d\rho(\xi), 
 \]
holds for a finite measure $\rho$ on $\T$ with the total mass $\alpha$, which implies that 
\[
|q(z)| \le |\beta| + \alpha \frac{1+|z|}{1-|z|}. 
\]
Consequently, we obtain 
\begin{equation}
\label{z-f_st}
|z - f_{st}(z)| \le 2|z|\left(\alpha\frac{1+|z|}{1-|z|} + |\beta|\right).
\end{equation}
Now making use of \cite[Theorem 1.6 in p.21]{Pom75} and \eqref{z-f_st} one obtains
\begin{align*}
|f_{s}(z) - f_{t}(z)|  &= |f_{s}(z) - f_{s}(f_{st}(z))| = \left|\int_{f_{st}(z)}^{z}f_{s}'(u)\, du\right|\\
&\le |z -f_{st}(z)|  |f_{s}'(0)|\dfrac{1+|z|}{(1-|z|)^{3}} \le 
 \frac{8\alpha|z|}{(1-|z|)^{4}} + \frac{4 |\beta| |z|}{(1-|z|)^{3}}, 
\end{align*}
which implies \eqref{eq:estimate_LC1}. For the last statement 
\eqref{eq:estimate_LC2}, note that $f_{s}'(0) = f_t'(0)/f_{st}'(0) \ge f_t'(0)$ by the Schwarz lemma and the simple inequality $(b-a)/(b+a)\le (1/a) - (1/b)$ holds for $0<a \le b \le 1$. 
\end{proof}


\end{document}